\documentclass[11pt]{amsart}
\usepackage{amscd,amssymb}
\usepackage{graphicx}
\usepackage[all]{xy}

\theoremstyle{plain}

\newtheorem{thm}[equation]{Theorem}
\newtheorem{prop}[equation]{Proposition}
\newtheorem{cor}[equation]{Corollary}

\newtheorem{lemma}[equation]{Lemma}

\numberwithin{equation}{section}

\newcommand{\Q}{\mathbb Q}
\newcommand{\Z}{\mathbb Z}
\newcommand{\R}{\mathbb R}
\newcommand{\C}{\mathbb C}

\def\Hom{{\rm Hom}}
\def\Aut{{\rm Aut}}

\def\SL{{\rm SL}}

\def\Spin{{\rm Spin}}

\def\GL{{\rm GL}}
\def\PGL{{\rm PGL}}

\def\Gal{{\rm Gal}}

\def\Stab{{\rm Stab}}

\def\Res{{\rm Res}}
\def\der{{\rm der}}

\def\R{{\mathbb R}}

\def\Z{{\mathbb Z}}

\title {Twisted Bhargava cubes}

\author{Wee Teck Gan and Gordan Savin}

\address{W.T.G.:   Department of Mathematics, National University of Singapore, 10 Lower Kent Ridge Road
Singapore 119076} \email{matgwt@nus.edu.sg}
\address{G. S.: Department of Mathematics, University of Utah, Salt Lake City, UT}\email{savin@math.utah.edu}
 
\begin{document}
\maketitle

\begin{abstract}
In his re-interpretation of Gauss's composition law for binary quadratic forms, Bhargava determined the integral orbits of a prehomogeneous vector space which arises naturally in the structure theory of the split group $\Spin_8$. 
We consider a twisted version of this prehomogeneous vector space which arises in quasi-split $\Spin_8^E$, where $E$ is an \'etale cubic algebra over a field $F$. We classify the generic orbits over $F$ by twisted composition $F$-algebras of $E$-dimension $2$.
\end{abstract}

\section{\bf Introduction}

In a seminal series of papers (\cite{B1}, \cite{B2},  \cite{B3}), Bhargava has extended Gauss's composition law for binary quadratic forms  to far more general situations. The key step in his extension is the investigation of the integral orbits of a group over $\Z$ on a lattice in a  prehomogeneous vector space. The prehomogeneous vector space which plays a role in elucidating the nature of the classical Gauss's composition arises from a simply connected Chevalley 
group $G$ of type $D_4$.  More precisely, let $P=MN$ be a maximal parabolic subgroup of $G$ corresponding to the branching point of the Dynkin diagram of type $D_4$. As it is readily seen from the Dynkin diagram,  the derived group $M_{\der}$ of the Levi factor $M$ is isomorphic to $\SL_2^3$. The unipotent radical $N$ is  
9-dimensional, and is a two step nilpotent group with 1-dimensional center $Z$.  The adjoint action of $M_{\der}$ on the abelian quotient $N/Z$ is isomorphic to 
$V=V_2 \otimes V_2 \otimes V_2$ where $V_2$ is the standard 2-dimensional representation of $\SL_2$. 
Since Bhargava regards an element of $(\Z^2)^{\otimes 3}$ as a cube whose vertices are labelled by elements of $\Z$, we shall refer to the prehomogeneous vector space  $V$ or its elements as Bhargava's cubes. 
\vskip 5pt

One of Bhargava's achievements is the determination of  the corresponding integral orbits, i.e.  $\SL_2(\Z)^3$-orbits on $\Z^2 \otimes \Z^2 \otimes \Z^2$. In particular, he discovered that  generic  orbits are in a bijection with isomorphism classes of tuples $(A, I_1, I_2, I_3)$ where 
 $A$ is an order in an \'{e}tale  quadratic $\Q$-algebra and $I_1$, $I_2$ and $I_3$ are elements in the narrow class group of $A$, i.e. invertible fractional ideals, such that $I_1 \cdot I_2 \cdot I_3 = 1$.  
 More precisely,  to every cube  Bhargava attaches three pairs $(A_i,B_i)$, $i=1,2,3$, of $2\times 2$ matrices by slicing the cube in the three possible ways. In this way he obtains three binary quadratic forms 
 \[ 
 Q_i(x,y)=-\det(A_i x + B_i y). 
 \] 
 A remarkable fact, discovered by Bhargava, is that the three forms have the same discriminant $\Delta$. It is a degree 4 polynomial on $V$, invariant under the action of $M_{\der}$. The cube is generic if $\Delta\neq 0$. In this case, the ring $A$ is the unique quadratic order of discriminant $\Delta$ and the three fractional ideals $I_i$ correspond to the three quadratic forms $Q_i$ by a dictionary that essentially goes back to Gauss. 
 
 \vskip 5pt 

We now consider the group $G$ over a field $F$ of characteristic different from $2$ and $3$. 
The group $G$ is exceptional in the sense that its outer automorphism group is isomorphic to $S_3$: no other absolutely simple linear algebraic group has such a large outer automorphism group. 
 In particular, since $S_3$ is also the group of automorphisms of the split \'{e}tale cubic $F$-algebra $F \times F \times F$, we see that every \'{e}tale cubic $F$-algebra $E$ determines a quasi-split form $G_E$. 
  Fixing an \'epinglage of $G$ defines a splitting of the outer automorphism group $S_3$ to ${\rm Aut}(G)$, so that   $S_3$ acts on $V$ by  a group of symmetries of the cube, fixing two opposite vertices.
  Then the quasi-split group $G_E$ contains a maximal parabolic subgroup $P_E= M_E  N_E$,  which is a twisted form of the parabolic $P$ mentioned above. The derived group $M_{E,\der}$ of $M_E$ is isomorphic to $\Res_{E/F}\SL_2$. The
 adjoint action of $M_{E,\der}$ on $N_E/Z_E$, where $Z_E$ is the center of $N_E$,  is isomorphic to a twisted form $V_E$ of $V$. We shall call $V_E(F)$ (or its elements) the $E$-twisted Bhargava cube.

\vskip 5pt
 
 Since the action of $S_3$ on $V$ permutes the three pairs $(A_i,B_i)$ of $2\times 2$ matrices obtained by slicing a cube in three different ways, it follows by Galois descent that $\Delta$ gives rise to a degree 4 polynomial invariant on $V_E$,  denoted by $\Delta_E$.  It is a quasi-invariant for $M_E$. More precisely, if $v\in V_E(F)$ and $g\in M_E(F)$, then 
 \[ 
 \Delta_E(gv)= \chi(v)^2 \cdot   \Delta_E(v) 
 \] 
 where $\chi$ is a character of $M_E$ given by the adjoint action on $Z_E$. 
 An $M_E(F)$-orbit  $\mathcal O\subset V_E(F)$ is called generic if $\Delta_E(v) \neq 0$ for one and hence for all $v\in \mathcal O$. If $\mathcal O$ is generic, then 
  the quadratic algebra $K=F(\sqrt{\Delta_E(v)})$ is  \'etale. It is an invariant of the generic orbit. 
\vskip 5pt 

The purpose of this paper is to classify the generic $M_E(F)$-orbits on $V_E(F)$.  The main result is:
\vskip 5pt

\begin{thm} \label{T:main} 
Let $F$ be a field of characteristic different from  $2$ or $3$. 
Fix an \'{e}tale cubic $F$-algebra $E$. 
\vskip 5pt

\noindent (i) There are natural bijections between the following sets:
\vskip 5pt

\begin{itemize} 

\item[(a)] Generic $M_E(F)$-orbits $\mathcal O$ on the $E$-twisted Bhargava cube.

\item[(b)]  $E$-isomorphism classes of $E$-twisted composition algebras $(C,Q,\beta)$ over $F$ which are of $E$-dimension $2$. 
 
\item[(c)]  $E$-isomorphism classes of pairs $(J,i)$ where $J$ is a Freudenthal-Jordan algebra over $F$ of dimension $9$ and
\[  i: E \hookrightarrow J \]
is an $F$-algebra homomorphism. Here an $E$-isomorphism from $(J,i)$ to $(J',i')$ is a commutative diagram
\[  \begin{CD}
E @>i>> J \\
@VVV  @VVV \\
E  @>i'>> J' \end{CD} \]
where the first vertical arrows is the identity, while the second is an $F$-isomorphism of $J$ and $J'$. 
 \end{itemize}
\vskip 5pt

\noindent (ii) The bijections in (i)   identify
\[ \Stab_{ M_E}( \mathcal{O}) \cong \Aut_E(C,Q,\beta)\cong  \Aut_E(i: E \hookrightarrow J). \]

\vskip 5pt

\noindent (iii)   Let $K=F(\sqrt{\Delta_E(v)})$ be the \'etale quadratic algebra $K$ attached to a generic orbit $\mathcal O$ containing $v$.  Let $L=E \otimes_F K$. The group 
$\Stab_{ M_E}( \mathcal{O})$ in (ii) sits in a short exact sequence of algebraic groups:
\[  \begin{CD}
1 @>>>  T_{E,K} @>>> \Stab_{ M_E}( \mathcal{O}) @>>> \Z/2\Z @>>>1 \end{CD} \]
where
\[   T_{E,K}(F) = \{ x \in L^{\times} :  N_{L/E}(x) = 1 = N_{L/K}(x) \} \]
is a 2-dimensional torus 
and  where the conjugation action of the nontrivial element of $\Z/2\Z$ on $T_{E,K}$ is given by  $x \mapsto x^{-1}$.
\end{thm}
 
\vskip 5pt

The reader is probably not familiar with some terminology in the theorem, so an explanation is  necessary.  
 In order to define twisted composition algebras,  recall that the algebra $E$ carries a natural 
cubic form, the norm $N_E$.  The norm defines a quadratic map $x\mapsto x^{\#}$ from $E$ to $E$ such that $x\cdot x^{\#}=N_E(x)$. 
For example, if $E=F^3$, then 
\[ 
N_E(x_1,x_2,x_3)=x_1 x_2 x_3 \text{ and } (x_1,x_2,x_3)^{\#}=(x_2x_3, x_3 x_1, x_1 x_2).
\]   

\vskip 5pt 
\noindent  Now, an $E$-twisted composition algebra (or simply twisted composition algebra)  of $E$-dimension 2 is a triple  $(C ,Q,\beta)$ where
\vskip 5pt

\begin{itemize}
\item $C$ is an $E$-vector space of dimension 2. 

\item $Q: C \longrightarrow E$ is a quadratic form.  
\vskip 5pt

\item $\beta: C \longrightarrow C$ is a quadratic map such that, for every $v\in C$ and $x\in E$, 
\[ \beta(x v) = x^{\#}\cdot  \beta(v)  \quad \text{and} \quad Q(\beta(v)) = Q(v)^{\#}.  \]
\vskip 5pt

\item If  $b_Q$ is the bilinear form associated to $Q$, then $b_Q(v,\beta(v)) \in F$ for every $v\in C$. 
 \end{itemize}

\vskip 5pt 
This definition is due to Springer, as is the bijection of the sets  (b) and (c). More precisely,  suppose  we have an algebra embedding 
$ i: E \hookrightarrow J$.  Then  we have a decomposition
\[  J = E \oplus C \]
where $C$ is defined as the orthogonal complement to $E$, with respect to the trace form on $J$. The upshot is that the Jordan algebra $J$ determines a 
structure of a twisted composition algebra on $C$, and vice versa. 
\vskip 5pt 

Our contribution is the bijection between the sets (a) and (b).  Starting with a twisted cube, we define a twisted composition algebra. 
In fact, the construction works over $\mathbb Z$, and can be tied to 
Bhargava's description as follows.  Let $(I_1,I_2,I_3)$ be a triple of ideals in a quadratic order $A$ such that $I_1\cdot I_2\cdot I_3 =A$. 
 Let $N(I)$ denote the norm of the ideal $I$, and $z\mapsto \bar z$ denote the action of the non-trivial automorphism of the 
 \' etale quadratic $\mathbb Q$-algebra containg $A$. Let 
\[  C =I_1 \oplus I_2\oplus I_3. \]
Then $C$ is  a twisted composition algebra with the quadratic form 
  $Q: C \rightarrow \mathbb Z \times \mathbb Z\times \mathbb Z$  defined by 
\[ 
Q(z_1,z_2,z_3)=\left(\frac{N(z_1)}{N(I_1)}, \frac{N(z_2)}{N(I_2)},\frac{N(z_3)}{N(I_3)}\right)  
\]
and the quadratic map   $\beta: C \rightarrow C$ defined by
\[ 
\beta(z_1,z_2,z_3)= (\bar z_2\bar z_3 N(I_1) , \bar z_3 \bar z_1 N(I_2),\bar z_1 \bar z_2N(I_3)).  
\]

\vskip 5pt 

The key parts of the paper are as follows. In order to prove the correspondence of  generic $M_E(F)$-orbits and twisted composition algebras, 
 we  give a Galois cohomological argument in Theorem \ref{T:main-orbit}, based on the observation that the stabiliser of a distinguished cube is isomorphic to the automorphism group of a distinguished twisted composition algebra. This gives a conceptual explanation for the existence of the bijection. However, for arithmetic applications (such as Bhargava's), it is essential to have an explicit description of the bijection.  This is done in two steps. Firstly, after reviewing the theory of twisted composition algebras, we prove  in Proposition  \ref{P:reduced-basis} that every twisted composition algebra $C$ of $E$-dimension 2 has a reduced basis, i.e. a basis of the form $\{v, \beta(v)\}$ for some $v\in C$. Secondly,  by re-interpreting Bhargava's work in the framework of twisted composition algebras in Section \ref{S:explicit}, we attach  to every generic $E$-twisted cube  a twisted composition algebra together with a good basis. In this correspondence, changing the cube by another in the same $M_E(F)$-orbit corresponds to changing the good basis. Since reduced bases are good, every twisted composition algebra is obtained in this construction. 
\vskip 5pt
 
  We also consider  $\tilde{M} =  M\rtimes S_3$ and its twisted form $\tilde M_E$. In this case generic 
 $\tilde M_E(F)$-orbits correspond to the $F$-isomorphism classes of objects in (b) and (c). 
The isomorphisms of the stabilizer groups in (ii) lead us to another description of $T_{E,K}$,  which we view as an {\em exceptional Hilbert 90} theorem. This is the topic of Section \ref{S:H90}. 
We conclude the paper by illustrating the main results in the case where $F$ is a local field.
 
\vskip 10pt 

\section{\bf \'{E}tale Cubic Algebras}

Let $F$ be a field of characteristic different from $2$ and $3$. Let $\bar F$ be a separable closure of $F$,  with the absolute Galois group  $\Gal(\overline{F}/F)$. 
\vskip 5pt

\subsection{\bf \'Etale cubic algebras.}

An \'{e}tale cubic algebra is an $F$-algebra $E$ such that $E \otimes_F \overline{F} \cong \overline{F}^3$. More concretely, an \'{e}tale cubic $F$-algebra is of the form:
\[ E = \begin{cases}
\text{$F \times F \times F$;} \\
\text{$F \times K$, where $K$ is a quadratic field extension of $F$;} \\
\text{ a cubic field.} \end{cases} \]
Since the split algebra $F \times F \times F$ has 
automorphism group $S_3$ (the symmetric group on 3 letters),
the isomorphism classes of \'{e}tale cubic algebras $E$ over $F$ are naturally classified by the 
pointed cohomology set $H^1(F, S_3)$, or more explicitly by the
set of conjugacy classes of  homomorphisms 
\[  \rho_E : \Gal(\overline{F}/F) \longrightarrow S_3. \]
\vskip 10pt

\subsection{\bf Discriminant algebra of $E$.} 

By composing the homomorphism $\rho_E$ with the sign character of $S_3$, we obtain a quadratic character (possibly trivial) of $\Gal(\overline{F}/F)$ which corresponds to an \'{e}tale quadratic algebra $K_E$. We call $K_E$ the {\em discriminant algebra} of $E$. To be concrete, 
\[  K_E = \begin{cases}
F \times F, \text{  if $E = F^3$ or a cyclic cubic field;} \\
K, \text{  if $E = F \times K$;} \\
\text{the unique quadratic subfield in the Galois closure of $E$ otherwise.} 
\end{cases} \] 

\vskip 10pt

\subsection{\bf Twisted form of $S_3$.}

Fix an \'{e}tale cubic $F$-algebra $E$. 
Then, via the associated homomorphism $\rho_E$, $\Gal(\overline{F}/F)$ acts on $S_3$ (by inner automorphisms) and thus define a twisted form
$S_E$ of the finite constant group scheme $S_3$.  For any commutative $F$-algebra $A$, we have
\[  S_E(A) = \Aut_A(E \otimes_F A). \]
\vskip 10pt

\subsection{\bf Quadratic map $\#$.} 
Given an \'etale cubic $F$-algebra,
let $N_E : E \longrightarrow F$ be the norm map on $E$ and let $Tr_E :  E \longrightarrow F$ be the trace map. 
Then $N_E$  is a cubic form and $Tr_E$ is a linear form on $E$. There is a quadratic map
\[  \#:  E \longrightarrow E \]
such that 
\[  a^{\#} \cdot a = a \cdot a^{\#}  = N_E(a) \quad \text{for $a \in E$.} \]
It has an associated symmetric bilinear map
\[  a \times b :=  (a+b)^{\#}  - a^{\#} - b^{\#}. \]
For the split algebra $F^3$, we have:
\[    N(a_1, a_2, a_3) = a_1a_2a_3, \quad  Tr(a_1,a_2, a_3)  = a_1+a_2+a_3, \quad (a_1, a_2, a_3)^{\#} = (a_2a_3, a_3a_1, a_1 a_2). \]

We note the following identity in $E$:
 \begin{equation} \label{E:curious}
    (f \times y) y + f y^{\#} = Tr_{E/F}(fy^{\#}).
 \end{equation}
 This curious identity can be checked in $E \otimes_F \overline{F} \cong \overline{F}^3$; we leave it as an interesting exercise for the reader. 
 \vskip 10pt

\section{\bf Twisted Composition Algebras}

In this section, we introduce the $E$-twisted composition algebra of dimension $2$ over $E$. This notion was introduced by Springer, and the two standard (perhaps only) references, covering many topics of this paper, are {\em The Book of Involutions} \cite{KMRT} and {\em Octonions, Jordan Algebras and Exceptional Groups} \cite{SV}.  Twisted composition algebras are treated in \cite[Chapter VIII, \S 36]{KMRT} and \cite[Chapter 4]{SV}.

\vskip 5pt

\subsection{\bf Twisted composition algebras.} 
A twisted composition algebra over $F$ is a quadruple $(E, C ,Q,\beta)$ where
\vskip 5pt

\begin{itemize}
\item $E$ is an \'etale cubic $F$-algebra;

\item $C$ is an $E$-vector space equipped with a nondegenerate quadratic form $Q$, with associated symmetric bilinear form $b_Q(v_1,v_2) = Q(v_1+v_2) - Q(v_1) -Q(v_2)$;
\vskip 5pt

\item $\beta: C \longrightarrow C$ is a quadratic map such that
\[ \beta(a v) = a^{\#}\cdot  \beta(v)  \quad \text{and} \quad Q(\beta(v)) = Q(v)^{\#},  \]
for every $a \in E$ and $v\in C$. 
\vskip 5pt

\item if we set
\[  N_C(v) :=  b_Q(v,\beta(v)), \]
then $N_C(v) \in F$, for every $v\in C$.  
 \end{itemize}
For a fixed $E$, we shall call $(C,Q, \beta)$ an $E$-twisted composition algebra (over $F$), and the cubic form $N_C$ the norm form of $C$.
Frequently, for ease of notation, we shall simply denote this triple by $C$, suppressing the mention of $Q$ and $\beta$. 
\vskip 5pt

\subsection{\bf Morphisms.}
Given twisted composition algebras $(E, C, Q, \beta)$ and $(E', C', Q', \beta')$, an $F$-morphism
of twisted composition algebras is a pair
$(\phi, \sigma) \in \Hom_F(C,C') \times \Hom_F(E, E')$ 
  such that 
  \[  \phi (av)  = \sigma(a) \cdot \phi(v) \]
  for $v \in C$ and $a \in E$, and  
 \[  \phi \circ \beta = \beta' \circ \phi \quad \text{and} \quad \sigma \circ Q  = Q' \circ \phi. \]
 In particular,we have  the automorphism group  $\Aut_F(E, C,Q,\beta)$.
 The second projection gives a natural homomorphism
\[ 
\Aut_F(E, C,Q,\beta) \rightarrow S_E. 
\]
The kernel of this map is the subgroup $\Aut_E(C, Q,\beta)$ consisting of those $\phi$ which are $E$-linear; we shall call these $E$-morphisms.  
\vskip 5pt

\subsection{\bf $\Aut_F(E,C)$-action and isomorphism classes.}  \label{SS:isom}
Let us fix an $E$-vector space $C$ and let $\Aut_E(C)$ be the automorphism group of $C$ as an $E$-vector space.  Let
\[  \Aut_F(E, C) = \{  (g, \sigma) \in \Aut_F(C) \times \Aut_F(E): g \circ \lambda  = \sigma(\lambda) \cdot g \, \text{for all $\lambda \in E$} \}. \]
This is the group of $E$-sesquilinear automorphisms of $C$. The second projection induces  a short exact sequence
\[  \begin{CD} 
1 @>>> \Aut_E(C)  @>>> \Aut_F(E,C) @>>> S_E @>>> 1. \end{CD} \]
This short exact sequence is split. Indeed, the choice of an $E$-basis for $C$ gives a splitting, with $S_E$ acting on the coordinates with respect to the basis. 
\vskip 5pt

Now if $(C, Q,\beta)$ is an $E$-twisted composition algebra, then for any $(g, \sigma) \in \Aut_F(E,C)$, the triple \[  (C', Q', \beta')  = (C, \sigma \circ Q \circ g^{-1},  g \circ \beta \circ g^{-1}) \]
 is also an $E$-twisted composition algebra. The norm forms are related  by:
\[  N_{C'}  =  N_C \circ g^{-1}. \]
Moreover, we have:
\[  (g ,\sigma) \in \Hom_F( (E, C, Q, \beta),  (E, C', Q',  \beta')). \]

\vskip 5pt

Thus the map $(Q, \beta) \mapsto (Q',\beta')$ defines an action of $\Aut_F(E, C)$ on the set of pairs $(Q, \beta)$ which define an $E$-twisted composition algebra structure on $C$.
 The orbits of such pairs under $\Aut_F(E, C)$ are precisely the $F$-isomorphism
classes of $E$-twisted composition algebra of a given $E$-dimension $\dim_E C$, and the stabiliser of a given pair $(Q,\beta)$ is precisely the automorphism group $\Aut_F(E,C, Q,\beta)$.
Similarly, 
the set of orbits under $\Aut_E(C)$ is the set of $E$-isomorphism classes of such $E$-twisted composition algebras, and the stabiliser of a particular $(Q,\beta)$ is $\Aut_E(C, Q,\beta)$.
\vskip 5pt
 
\subsection{\bf Dimension $2$ case.}   \label{SS:example}
 It is known, by Corollary 36.4 in \cite{KMRT},  that for any $E$-twisted composition algebra $(C,Q,\beta)$, $\dim_E C= 1$, $2$, $4$ or $8$. 
 We shall only be interested in the case when $\dim_E C = 2$.  
 \vskip 5pt

   We give an example that will feature prominently in this paper. 
 We set $C_E= E\oplus E$, and define $Q$ and $\beta$ by 
 \[ 
 Q(x, y) =x\cdot y \text{ and } \beta(x,y)=(y^{\#}, x^{\#}) 
 \] 
 for every $(x,y)\in E\oplus E$.  
 It is easy to check that this defines an $E$-twisted composition algebra over $F$, with the norm form
 \[   N_C(x, y) = N_E(x) + N_E(y). \]
\vskip 5pt

 The group of automorphisms of this $E$-twisted composition algebra is easy to describe. 
 Let $E^1$ be the set of elements $e$ in $E$ such that $N(e)=e\cdot e^{\#}= 1$. For every element $e\in E^1$ we have an
 $E$-automorphism $i_e$ defined by $i_e(x,y)=(ex,e^{\#}y)$. We also have an $E$-automorphism $w$ defined by 
 $w(x,y)=(y,x)$. The group of $E$-automorphisms is 
 \[ 
 \Aut_E(C_E, Q,\beta) = E^1 \rtimes \Z/2\Z 
 \] 
 and the group of $F$-automorphisms is 
 \[ 
 \Aut_F(C_E, Q,\beta) = (E^1 \rtimes \Z/2\Z )\rtimes S_E = E^1 \rtimes (\Z/2\Z \times S_E). 
 \] 
If $E=F\times F\times F$, we denote the corresponding twisted composition algebra by $C_0 = (C_0, Q_0, \beta_0)$ and refer to it as the split twisted composition algebra. In this case, $E^1$ consists of $(t_1,t_2,t_3)$ such that  $t_1 t_2 t_3=1$, so that 
\[  \Aut_{E}(C_0, Q_0, \beta_0) \cong \mathbb{G}_m^2 \rtimes \Z/2\Z. \]
Observe that there is a natural splitting
\begin{equation} \label{E:split} 
  S_3 \times \Z/2\Z \longrightarrow   \Aut_F(C_0, Q_0, \beta_0). \end{equation}
  
 \vskip 10pt
\subsection{\bf Identitites.}
It follows by   \cite[Prop. 36.3]{KMRT} that if $(E, C, Q,\beta)$ is a twisted composition algebra over $F$, then $ C \otimes_F \overline{F}$ is isomorphic to $C_0 \otimes_F \overline{F}$. 
This fact is useful for verifying polynomial identities in $C$. Indeed, any polynomial identity in $C$ may be verified over $\overline{F}$ and thus just needs to be checked in $C_0$. In the following lemma, we list some useful identities which may be checked in this manner.
\vskip 5pt

\begin{lemma} \label{L:ID}
Let $(E, C,Q,\beta)$ be a twisted composition algebra over $F$. Then one has
\begin{equation} \label{E:beta2}
\beta^2(v)  = N_C(v) v - Q(v) \beta(v), \end{equation}
and
\begin{equation} \label{E:beta-reduced}
  \beta(x v + y \beta(v))  = (y^{\#} N(v) - ( -Q(v) x) \times y)\cdot  v + (x^{\#}  - Q(v)y^{\#}) \cdot \beta(v) \end{equation}
for any $v \in C$ and $x,y, \in E$. 
\end{lemma}
 
 \vskip 5pt
 
 It follows from (\ref{E:beta2}) that $Q$ is in fact determined by $\beta$ in a twisted composition algebra. 
 The proof of these identities can be found in \cite[Lemmas 4.1.3 and  4.2.7]{SV}. It is interesting to note that (\ref{E:beta-reduced})  looks slightly different from that in \cite[Lemma 4.2.7]{SV}, but is equivalent to the form we present here by the identity (\ref{E:curious}).
 \vskip 5pt
 
  \subsection{\bf Reduced basis.}  \label{S:reduced-basis}
If $\dim_EC =2$, we call an $E$-basis of $C$ of the form $\{v , \beta(v) \}$ a {\rm reduced basis} of $C$. We note:
 \vskip 5pt
 
 \begin{prop}  \label{P:reduced-basis}
Let $(C,Q,\beta)$ be an $E$-twisted composition algebra.
\vskip 5pt

\noindent (i) For $v \in C$, let 
\[ \Delta_C(v) = N_C(v)^2- 4 \cdot N_E(Q(v)) \in F. \]
 Then $\{ v, \beta(v)\}$ is an $E$-basis of $C$ if and only if $\Delta_C(v) \ne 0$. 
\vskip 5pt

\noindent (ii)  The degree $6$ homogeneous polynomial $\Delta_C$ factors over $F$ as:
\[  \Delta_C =  a \cdot P^2 \]
with $a \in F^{\times}$ and $P$ an absolutely irreducible homogeneous polynomial of degree $3$ over $F$.
The square class of $a$ is uniquely determined, 
and for any $g \in \Aut_E(C,Q,\beta)$,  \begin{equation} \label{E:P}
  P(gv)  =  \begin{cases}
   P(v),\text{  if $g \in \Aut_E(C,Q,\beta)^0$;} \\
   - P(v), \text{  if $g \notin  \Aut_E(C,Q,\beta)^0$.} \end{cases}
   \end{equation}
 \vskip 5pt
 
\noindent (iii) The algebra $(C, Q,\beta)$ has a reduced basis. 
\vskip 5pt

\noindent (iv) Let $\{v' , \beta(v')\}$ be another reduced basis of $C$. Let $g \in \Aut_E(C)$ be such that $g(v) = v'$ and $g(\beta(v)) = \beta(v')$. Then $\det(g)  \in F^{\times}$.
 
\end{prop}
\vskip 5pt

 \begin{proof}
 (i)  The set $\{v, \beta(v) \}$ is linearly independent if and only if the matrix of the symmetric bilinear form $b_Q$ 
 with respect to $\{v, \beta(v)\}$ has nonzero determinant.  Since
 \[  b_Q(v,v)  = 2Q(v), \quad  b_Q(\beta(v),\beta(v)) = 2 Q(v)^{\#} \quad \text{and} \quad b_Q(v,\beta(v)) = N_C(v), \]
 it follows that the determinant  is $-\Delta_C(v)$.
 \vskip 5pt
 
\noindent (ii) We first work over $\overline{F}$, in which case we may assume that
 $C = E^2$, with $E =\overline{F}^3$,  $Q(x,y)  = xy$ and $\beta(x,y)  = (y^{\#}, x^{\#})$. 
 Then  $N_C(x,y)  = N_E(x)  + N_E(y)$.
So 
\[  \Delta_C(x,y)  =  (N_E(x) + N_E(y))^2 - 4 N_E(x)  N_E(y) = (N_E(x)  - N_E(y))^2. \]
The cubic polynomial  $P_0(x,y) = N_E(x)  - N_E(y) = x_1x_2x_3 - y_1y_2y_3$ (with $x = (x_1,x_2,x_3) \in \overline{F}^3$) is easily seen to be irreducible over $\overline{F}$. 
\vskip 5pt

To descend back to $F$,  we note that for any $\sigma \in {\rm Gal}(\overline{F}/F)$,  $\sigma(P_0)  = \pm P_0$  by unique factorisation of polynomials over $\overline{F}$.  Thus there is a quadratic character $\chi_K$ of ${\rm Gal}(\overline{F}/F)$ such that $\sigma(P_0) =  \chi_K(\sigma) \cdot P_0$. If $K$ is the quadratic \'etale $F$-algebra associated to $\chi_K$, represented by $a \in F^{\times}$, then  we see that $P = \sqrt{a}^{-1} \cdot P_0$ is defined over $F$ and $\Delta_C= a \cdot P^2$. 
\vskip 5pt

It is clear that the square class of $a$ is uniquely determined. The equation (\ref{E:P}) can be checked over $\overline{F}$; we leave it to the reader.
\vskip 10pt

\noindent (iii)  Since $F$ has more than 3 elements (as we assumed that ${\rm char}(F) \ne 2$ or $3$), there exists $v \in C$ such that $P(v) \ne 0$.
 Hence $\Delta_C(v)  \ne 0$ by (ii) and $\{v, \beta(v) \}$ is a reduced basis by (i).

 \vskip 10pt
 \noindent (iv)  If $v' = xv + y\beta(v)$, then $\beta(v')$ is given by (\ref{E:beta-reduced}). So the transition matrix between the bases $\{v, \beta(v) \}$ and $\{v', \beta(v') \}$ is given by:
 \[  g = \left( \begin{array}{cc}
 x & y^{\#} N_C(v) - ( -Q(v) x) \times y \\
 y & x^{\#}  - Q(v)y^{\#} \end{array} \right). \]
 Hence
 \begin{align}
  \det(g) &= N_E(x)  - N_E(y) N_C(v) + ( -Q(v) x y^{\#}+  ((-Q(v)x) \times y) y) \notag \\
  &= N_E(x)  - N_E(y) N_C(v)  - Tr_E(Q(v)xy^{\#}) \in F \notag
  \end{align}
  where the second equality follows by applying (\ref{E:curious}).  
 \end{proof}
\vskip 5pt

We note that Proposition \ref{P:reduced-basis}(i) and (iii) is contained in \cite[Lemma 4.2.12]{SV}, but (ii) seems to be new; at least we are not able to find it in \cite{SV} or \cite{KMRT}. The results of the proposition will be used later in the paper. 
\vskip 5pt

\subsection{\bf The quadratic algebra $K_C$.} 
An immediate consequence of the proposition is that to very twisted composition algebra $(E,C,Q,\beta)$ with $\dim_E C =2$, we can associate an \'etale quadratic algebra $K_C$ which is given by the square-class of $\Delta_C(v) \in F^{\times}$ as in the proof of Proposition \ref{P:reduced-basis}(ii). 
Thus we have a map
\begin{equation} \label{E:K_C}
  \{ \text{twisted composition $F$-algebras with $E$-rank $2$} \} \longrightarrow \{ \text{\'etale quadratic $F$-algebras} \}. \end{equation} 
\vskip 5pt

For example,  if $C_E$ is the twisted composition algebra introduced in \S \ref{SS:example}  then 
\[\Delta_C(x,y)  =  (N_E(x) - N_E(y))^2\] and the quadratic algebra associated to $C_E$  is the split algebra $F \times F$. 
\vskip 5pt

  \vskip 10pt

 \subsection{\bf Cohomological description.}
 We come now to the classification of twisted composition algebras $C$  of rank $2$ over $E$.  
  Since every such $C$ is isomorphic to $C_0$ over $\overline{F}$, 
 the set of isomorphism classes of twisted composition algebras over $F$
is classified by the pointed cohomology set
\[  H^1( F, \Aut_F(F^3, C_0, Q_0, \beta_0)).\]
We have seen that  $\Aut_F(F^3, C_0, Q_0, \beta_0) \cong \mathbb{G}_m^2 \rtimes (\Z/2\Z \times S_3)$, and so 
 there is a natural map
\begin{equation}  \label{E:eqc}
 H^1( F, \Aut_F(F^3, C_0, Q_0, \beta_0)) \longrightarrow H^1(F, \Z/2\Z) \times H^1(F, S_3). \end{equation}
Composing this with the first or second projection, we obtain natural maps 
\begin{equation} \label{E:eq}
 H^1( F, \Aut_F(F^3, C_0, Q_0, \beta_0)) \longrightarrow H^1(F, \Z/2\Z)  = \{ \text{\'etale quadratic $F$-algebras} \}. \end{equation}
and
\begin{equation} \label{E:ec}
 H^1( F, \Aut_F(F^3, C_0, Q_0, \beta_0)) \longrightarrow H^1(F, S_3). \end{equation}  
All these projection maps are surjective, because of the natural splitting in (\ref{E:split}). 
Indeed, (\ref{E:split}) endows each fiber of the maps in (\ref{E:eqc}), (\ref{E:eq}) and (\ref{E:ec}) with a distinguished point. We shall see in a moment that  the map in (\ref{E:eq}) is the map defined in (\ref{E:K_C}). 
 \vskip 5pt

 For an \'etale cubic F-algebra $E$ with associated cohomology class $[E] \in H^1(F, S_3)$,   
  the fiber of (\ref{E:ec})  over $[E]$ is precisely the set of $F$-isomorphism classes of $E$-twisted composition algebras. Moreover, a Galois descent argument shows that the distinguished point in this fiber furnished by the splitting (\ref{E:split}) is none other than
the $E$-twisted composition algebra $C_E$ constructed in \S \ref{SS:example}.  
\vskip 5pt

Using $C_E$ as the base point, the fiber in question is identified naturally with
the set $H^1(F, \Aut_E(C_E, Q, \beta))$ modulo the natural action of $S_E(F)$ (by conjugation). 
The cohomology set $H^1(F, \Aut_E(C_E, Q, \beta))$ classifies the $E$-isomorphism classes of   $E$-twisted composition algebras $C$  over $F$, and the action of $S_E(F)$  is given by 
\[  \sigma:  (C,Q, \beta) \mapsto (C \otimes_{E, \sigma} E,  \sigma \circ Q,  \beta) \]
 for $\sigma \in S_E(F)$.
  \vskip 5pt

\begin{lemma}
The maps defined by (\ref{E:K_C}) and (\ref{E:eq}) are the same.
\end{lemma}

\begin{proof}
We fix the cubic algebra $E$ and let $C_E = (E^2, Q, \beta)$ be the distinguished $E$-twisted composition algebra introduced in \S \ref{SS:example}. Let $\Delta_C =  P^2$  be the homogeneous polynomials as given in Proposition  \ref{P:reduced-basis}(ii).
\vskip 5pt

Any $E$-twisted composition algebra $C'$ is given by a pair of tensors $(Q', \beta')$ on $E^2$ and there is an element $g \in \GL_2( E \otimes_F \overline{F})$ such that $g \cdot (Q, \beta)  = (Q', \beta')$. A 1-cocycle associated to $(Q', \beta')$ is given by
\[  a_{\sigma}  = g^{-1} \sigma(g) \in \Aut_{\overline{F}^3}(E^2,Q,\beta) \quad \text{ for $\sigma \in {\rm Gal}(\overline{F}/F)$.} \]
The corresponding $\Delta_{C'}$ is related to $\Delta_C$ by 
\[  \Delta_{C'}(v)  = \Delta_C(g^{-1} v). \]
\vskip 5pt

Now the quadratic algebra associated to $C'$ by (\ref{E:eq}) corresponds to the quadratic character 
\[  \chi: \sigma \mapsto  [a_{\sigma}]  \in \pi_0(\Aut_{\overline{F}^3}(E^2, Q,\beta))  = \Z/2\Z \]
of ${\rm Gal}(\overline{F}/F)$. By (\ref{E:P}), we thus have
\[  P(a_{\sigma}^{-1}  v)  = \chi(\sigma) \cdot P(v)  \]
for any $v \in (E \otimes_F \overline{F})^2$.

\vskip 5pt
On the other hand,  the quadratic algebra associated to $C'$ by (\ref{E:K_C}) is defined by $\sqrt{\Delta_{C'}(v)}$ for  any $v \in E^2$ such that $\Delta_{C'}(v) \ne 0$.  Since 
\[  \sqrt{\Delta_{C'}(v)}   = \sqrt{\Delta_C(g^{-1}v)}  = P(g^{-1}v), \]
we need to show that
\[  \sigma(P(g^{-1} v))  = \chi(\sigma) P(g^{-1}v). \]
But we have
\[  \sigma(P(g^{-1}v)) = P(\sigma(g)^{-1} v) = P(a_{\sigma}^{-1} g^{-1}v) = \chi(\sigma) \cdot P(g^{-1}v), \]
as desired.  This proves the lemma.
 \end{proof}
 \vskip 5pt

\subsection{\bf Tits construction.} Given an element 
\[  ([E], [K]) \in H^1(F, S_3) \times H^1(F, \Z/2\Z),\]
 we describe the composition algebras in the fiber of (\ref{E:eqc}) over $([E], [K])$.
  Note that by (\ref{E:split}), we have a distinguished point in this fiber. Now we have:

 \vskip 5pt

  \begin{prop}  \label{P:kmrt-quad} If $C$ is an $E$-twisted composition algebra, with the associated \'etale quadratic algebra $K$, 
  then we may identify $C$ with $E \otimes_F K$, such that 
\[  Q(x)  = e \cdot N_{E \otimes_F K /E}(x) \quad \text{ for some $e \in E^{\times}$} \]
and 
 \[  \beta(x) =  \overline{x}^{\#} \cdot e^{-1} \cdot \overline{\nu} \quad \text{ for some $\nu \in K$} \]
where $x \mapsto \overline{x}$ is induced by the nontrivial automorphism of $K$ over $F$. 
Moreover, we have:
\[  N_{E/F}(e)  =  N_{K/F}(\nu). \]
The distinguished point in the fiber of (\ref{E:eqc}) over $([E], [K])$ corresponds to taking $(e, \nu) = (1,1)$. 
 \end{prop}
\vskip 5pt

\begin{proof}
This proposition essentially follows by Galois descent. 
Indeed, a Galois descent argument, starting from the algebra $C_E$ introduced in \S \ref{SS:example}, shows that  $C$ can be identified with $E\otimes_F K$. Then $Q$ is necessarily of the form $e \cdot N_{E \otimes_F K/E}$ for some $e \in E^{\times}$. On the other hand,  we claim that for $x \in E \otimes_F K$ and $x_0 \in C$, one has
\[  \beta(x \cdot x_0) = \overline{x}^{\#} \cdot \beta(x_0). \]
Indeed, one can check this by going to $\overline{F}$, whence one is reduced to checking this identity in the split algebra $C_0$, where it is straightforward.  
This shows that $\beta$ is determined by $\beta(1) = \overline{\nu} \cdot e^{-1}$ for some $\nu \in E \otimes_F K$. However, the identity
\[  Q(1)^{\#}  = Q(\beta(1)) \]
implies that 
\[  \nu \cdot \overline{\nu} = N_{E \otimes_F K/E}(\nu) = N_{E/F}(e) \in F.  \]
The requirement that $N(x) \in F$ for all $x \in E \otimes_F K$ implies that
\[  Tr_{E \otimes_F K/E}( \overline{\nu} \cdot N_{E \otimes_F K/K}(x) )  \in F. \]
In particular, taking $x = 1$ and then a trace zero element $\delta \in K$ one obtains, respectively: 
\[  \nu + \overline{\nu}  \in F \quad \text{and} \quad \nu \delta + \overline{\nu \delta } \in F. \]
All these conditions imply that $\nu \in K$.
\vskip 5pt

Finally, it is easy to see by Galois descent that the distinguished point in the fiber over $([E] , [K])$
corresponds to $(e, \nu) = (1,1)$. 
\end{proof}

\vskip 5pt

The description of twisted composition algebras given in the above proposition is sometimes referred to as a Tits construction (though usually this terminology is reserved for the Jordan algebra associated to the above twisted composition algebra by Springer's construction, which is the subject matter of the next section). 

\noindent 
\vskip 5pt

\subsection{\bf Automorphism group.}
Using Proposition \ref{P:kmrt-quad}, it is not difficult to determine the automorphism group of any twisted composition algebra $C$. Indeed, if 
$C \cong  E \otimes_F K$ as in the proposition,  then the special orthogonal group
 \[  SO(C,Q) = \{ \lambda \in E \otimes_F K: N_{E \otimes K/E}(\lambda) = 1 \} \]
 acts $E \otimes K$-linearly on $C$ by multiplication and preserves $Q$. An element $\lambda \in SO(C,Q)$ preserves $\beta$ if and only if
 \[  \overline{\lambda}^{\#}   =  \lambda. \]
But $\lambda^{\#} = \lambda^{-1}$ since $N_{E \otimes K/E}(\lambda)  = \lambda \cdot \lambda^{\#} =1$. 
So 
\[  \Aut_E(C, Q, \beta) \cap SO(C,Q)  = \{ \lambda \in L = E \otimes K:  N_{L/E}(\lambda) =1 = N_{L/K}(\lambda) \} = T_{E,K}, \]
which is a 2-dimensional torus.
 Since we know the automorpshim group of the split twisted composition algebra $(C_0, Q_0, \beta_0)$, we see that
 \[  \Aut_E(C, Q,\beta)^0  =   T_{E,K} \]
 and $\Aut_E(C,Q,\beta)$ sits  in short exact sequences of algebraic groups as in Theorem \ref{T:main} (iii).

\vskip 5pt 

\subsection{\bf Cohomology of $T_{E,K}$.}
Using Proposition \ref{P:kmrt-quad} and the above description of $\Aut_E(C,Q,\beta)^0$, we can describe the fiber of the natural map
\[ H^1( F, \Aut_F(F^3, C_0, Q_0, \beta_0)) \longrightarrow H^1(F, \Z/2\Z) \times H^1(F, S_3) \]
over the element $([K] , [E]) \in  H^1(F, \Z/2\Z) \times H^1(F, S_3)$. 
Indeed, this fiber is equal to  
\[  \text{$H^1(F, T_{E,K})$ modulo the action of $S_E(F) \times \Z/2\Z$.} \]
 The cohomology group $H^1(T_{E,K})$  classifies twisted composition algebras with fixed $E$ and $K$, up to $E \otimes_F K$-linear isomorphism. With $L  = E \otimes_F K$, one has a short exact sequence of algebraic tori
\[  \begin{CD}
1 @>>> T_{E,K} @>>>  L^{\times}  @>N_{L/E} \times N_{L/K}>>   (E^{\times} \times K^{\times})^0 @>>> 1 \end{CD} \]
where 
\[  (E \times K)^0  = \{ (e, \nu) \in E^{\times} \times K^{\times}:  N_{E/F}(e) = N_{K/F}(\nu) \}. \]
The associated long exact sequence gives
\begin{equation}  \label{E:coho-T}
   H^1(F, T_{E,K})  \cong (E^{\times} \times K^{\times})^0/ {\rm Im} L^{\times}. \end{equation}
This isomorphism is quite evident in the context of Proposition  \ref{P:kmrt-quad}. Indeed, Proposition   \ref{P:kmrt-quad} tells us that any twisted composition algebra $C$ with invariants $(E,K)$ is given by an element $(e, \nu) \in (E^{\times} \times K^{\times})^0$. Any $L$-linear map from $C$ to another twisted composition algebra $C'$ with associated pair $(e' ,\nu')$ is given by multiplication by an element $a \in L^{\times}$, and this map is an isomorphism of twisted composition algebras if and only if 
\[  (e,\nu)  = (e' \cdot N_{L/E}(a),  \nu' \cdot N_{L/K}(a)).  \]
This is precisely what (\ref{E:coho-T}) expresses.
 
 \vskip 10pt

\section{\bf Springer's construction.}
We can now relate twisted composition algebras to Freudenthal-Jordan algebras. This construction is due to Springer. 
Our exposition follows  \S 38A, page  522,  in \cite{KMRT}.  
\vskip 5pt

\subsection{\bf Freudenthal-Jordan algebra of dimension $9$.}
A Freudenthal-Jordan algebra $J$ of dimension $9$ over $F$ is a Jordan algebra which is isomorphic over $\overline{F}$ to the Jordan algebra $J_0$ associated to the associative algebra $M_3(F)$ of $3 \times 3$-matrices, with Jordan product
\[  a \circ b = \frac{1}{2} \cdot (ab+ba). \]
 An element $a\in J$ satisfies a characteristic polynomial 
\[ 
X^3 - T_J(a) X^2 + S_J(a) X - N_J(a)\in F[X]. 
\] 
The maps $T_J$ and $N_J$ are called the trace and norm maps of $J$ respectively. The element 
\[ 
a^{\#}= a^2-T_J(a) a +S_J(a) 
\] 
is called the adjoint of $a$. It satisfies $a\cdot a^{\#}= N_J(a)$. The cross product of two elements $a,b\in J$ is defined by 
\[ 
a\times b =(a+b)^{\#} - a^{\#} - b^{\#}. 
\] 

\vskip 5pt

\subsection{\bf Cohomological description.}
The automorphism group of $J_0$ is $\PGL_3 \rtimes \Z/2\Z$, with $g \in \PGL_3$ acting by conjugation and the nontrivial element of $\Z/2\Z$ acting by the transpose: $a \mapsto a^t$. Thus, the isomorphism classes of Freudenthal-Jordan algebra of dimension $9$ is parametrized by the pointed set $H^1(F, \PGL_3 \rtimes \Z/2\Z)$, and there is a  exact sequence of pointed sets
\[  \begin{CD}
H^1(F, \PGL_3) @>f>>  H^1(F, \PGL_3 \rtimes \Z/2\Z) @>\pi>> H^1(F, \Z/2\Z)\\
 & & & &@| \\
  & &  &  & \{ \text{\'etale quadratic $F$-algebras}\}.  \end{CD} \]
The map $\pi$ is surjective 
 and the fiber of $\pi$ over the split quadratic algebra $F^2$ is the image of $f$. By  Proposition 39(ii) and Corollary 1 in \cite{Se}, page 52, the image of $f$ is  $H^1(F, \PGL_3)$ modulo a natural action of $\Z/2\Z$ (see \cite{Se}, page 52). Now the set $H^1(F, \PGL_3)$ parametrizes the set of central simple $F$-algebras $B$ of degree $3$, and the $\Z/2\Z$ action in question is $B \mapsto B^{op}$. Then the map $f$  sends  $B$ to the associated Jordan algebra.
\vskip 5pt

In general,  
for any \'etale quadratic $F$-algebra $K$, an element in the fiber of $\pi$ over $[K] \in H^1(F, \Z/2\Z)$ is the Jordan algebra $J_3(K)$ of $3 \times 3$-Hermitian matrices with entries in $K$. 
The automorphism group of $J_3(K)$ is an adjoint group $PGU_3^K \rtimes \Z/2\Z$.
Using $J_3(K)$ as the base point,  the fiber of $\pi$  over $[K] $ can then be identified with $H^1(F, PGU_3^K)$ modulo the action of $\Z/2\Z$ 
(by \cite{Se}, pages 50 and  52). By \cite{KMRT}, page 400, $H^1(F, PGU_3^K)$ has an interpretation as the set of isomorphism classes of pairs $(B_K, \tau)$ where 
\vskip 5pt
\begin{itemize}
\item $B_K$ is a central simple $K$-algebra of degree $3$,
\item  $\tau$ is an involution of the second kind on $B_K$,
\end{itemize}
Moreover the action of the non-trivial element $\tau_K \in \Aut(K/F) = \Z/2\Z$ is via the Galois twisting action: $B \mapsto B \otimes_{K, \tau_K}  K$,
 so that 
 \[  H^1(F, PGU_3^K) / \Z/2\Z \longleftrightarrow  \{ \text{$F$-isomorphism classes of $(B_K, \tau)$} \}.\] 
 Then the map $f$ sends $(B_K, \tau)$ to the Jordan algebra $B_K^{\tau}$ of $\tau$-symmetric elements in $B_K$. 
\vskip 5pt

If $J$ is a Freudenthal-Jordan algebra of dimension $9$, we will write $K_J$ for the \'etale quadratic algebra corresponding to $\pi(J)$. 
\vskip 5pt

\subsection{\bf Relation with twisted composition algebras.}
Fix an \'etale cubic $F$-algebra $E$ and a Freudenthal-Jordan algebra $J$.
Suppose  we have an algebra embedding
\[  i: E \hookrightarrow J. \]
Then, with respect to the trace form $T_J$, we have an orthogonal decomposition
\[  J = i(E) \oplus C \]
where $C = i(E)^{\perp}$. We shall identify $E$ with its image under $i$. Then 
for $e \in E$ and $v \in C$, one can check that $e \times v \in C$. Thus, setting
\[  e \circ v :=  -e \times v \]
 equips $C$ with the structure of an $E$-vector space. Moreover,
 writing
 \[  v^{\#} = ( - Q(v) , \beta(v)) \in E \oplus C =J \]
 for $Q(v) \in E$ and $\beta(v) \in V$, we obtain a quadratic form $Q$ on $C$ and a quadratic map $\beta$ on $C$. Then, by Theorem 38.6 in \cite{KMRT},
  the triple $(C,Q,\beta)$ is an $E$-twisted composition algebra over $F$. 
\vskip 5pt 

Conversely, given an $E$-twisted composition algebra $C$ over $F$, Theorem 38.6 in \cite{KMRT}  says that the space $E \oplus C$ can be 
given the structure of a Freuthendal-Jordan algebra over $F$. 
In particular,  we have described the bijective correspondence between the objects in (b) and (c) of the main theorem:
\[  \{\text{$E$-twisted composition algebras over $F$} \} \]
\[  \updownarrow \]
\[   \{ \text{$i: E \longrightarrow J$ with $J$ Freudenthal-Jordan of dimension $9$} \}. \]

\vskip 5pt
\noindent It is also clear that under this identification, one has
\[  \Aut_F( i: E \rightarrow J)    = \Aut_F ( i(E)^{\perp}). \]

\vskip 5pt
\subsection{\bf  Example.} \label{SS:example2} 
Let $K$ be an \'etale quadratic $F$-algebra and consider the Jordan algebra $J_3(K)$ of $3\times 3$ Hermitian matrices with entries in $K$.
Let $E=F\times F\times F$ be the subalgebra of $J_3(K)$ consisting of diagonal matrices.  Then $C$ consists of matrices 
\[  v=\left( \begin{array}{ccc}
0 & \bar z_3 & z_2 \\
z_3 & 0 & \bar z_1 \\
 \bar z_2 & z_1 & 0 \end{array}  \right).  \]
 Thus $C =K\times K \times K$, and one cheeks that 
 \[ 
Q(z_1,z_2,z_3)=(z_1\bar z_1, z_2\bar z_2,  z_3\bar z_3) 
\] 
and 
\[ 
\beta(z_1,z_2,z_3)=(\bar z_2\bar z_3, \bar z_3\bar z_1, \bar z_1\bar z_2). 
\] 
\vskip 5pt

\noindent The algebra $C$ is  the distinguished point in the fiber of $([F^3], [K])$, in the sense of Proposition \ref{P:kmrt-quad}. 
The automorphism group of $C$ is given by
\[  \Aut_F(C, Q,\beta) = (K^1 \times K^1 \times K^1)^0 \rtimes (\Z/2\Z \times S_3) \]
where $K^1$ denotes the torus of norm 1 elements in $K$ and $(K^1 \times K^1 \times K^1)^0$ denotes the subgroup of triples $(t_1, t_2, t_3)$ such that $t_1t_2t_3  =1$.   
  \vskip 5pt

\subsection{\bf The quadratic algebra associated to $i:E \rightarrow J$.} 
If an $E$-twisted composition algebra $C$ corresponds to a conjugacy class of  embeddings $i: E \longrightarrow J$, then we may ask how the quadratic algebra $K_C$ associated to $C$  can be described 
in terms of $i: E  \longrightarrow J$.  In this case, $C = E^{\perp}$ is an $E$-twisted composition algebra, and so $C = E \otimes K_C$ for a quadratic algebra $K_C$ as in  Proposition \ref{P:kmrt-quad}. On the other hand, we know that $J$ is associated to a pair $(B_{K_J}, \tau)$, where $B_{K_J}$ is a central simple algebra over an \'etale quadratic $F$-algebra $K_J$ and $\tau$ is an involution of the second kind.
 Now, Examples (5) and (6) on page 527 in \cite{KMRT} show that 
 \[  [K_C] \cdot [K_E] \cdot [K_J] = 1 \in H^1(F, \Z/2\Z) = F^{\times}/ F^{\times 2}. \] 

 \vskip 10pt

\section{\bf Quasi-split Groups of type $D_4$} 

In this section, we shall introduce the $E$-twisted Bhargava's cube by way of the quasi-split groups
 of type $D_4$.
 \vskip 5pt

\subsection{\bf Root system.}
Let $\Psi$ be a root system of type $D_4$, and $\Pi=\{\alpha_0, \alpha_1,\alpha_2,\alpha_3\}$ be a set of simple roots such that the 
corresponding Dynkin diagram is 

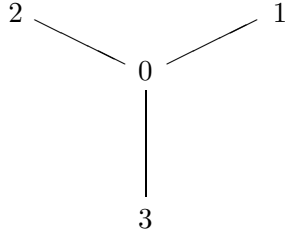
\begin{picture}(100,130)(-180,0) 

\put(19,18){3} 

\put(22,30){\line(0,1){40}}

\put(19,74){0}

\put(30,80){\line(2,1){34}}
\put(-20,97){\line(2,-1){34}}

\put(-30,96){2}
\put(70,96){1}

\end{picture}

\noindent The group of diagram automorphisms $\Aut(\Pi)$ is identified with $S_3$, the group of permutations of $\{1,2,3\}$. We denote the highest root by $\beta_0 = \alpha_1+ \alpha_2 +\alpha_3 + 2 \alpha_0$.
\vskip 5pt

\subsection{\bf Quasi-split groups of type $D_4$.} 
Let $G$ be a split, simply connected Chevalley group of type $D_4$. We fix a maximal torus $T$ contained in a Borel subgroup $B$ defined over $F$. The group $G$ is then generated by root groups $U_{\alpha}\cong \mathbb G_a$, where $\alpha\in\Psi$. Steinberg showed that one can pick the isomorphisms 
$x_{\alpha}: \mathbb G_a \rightarrow U_{\alpha}$ such that 
\[ 
[x_{\alpha}(u), x_{\alpha'}(u')]= x_{\alpha+\alpha'}(\pm uu')
\] 
whenever $\alpha+\alpha'$ is a root.  Fixing such a system of isomorphisms is fixing an \'epinglage (or pinning) for $G$. As Kac noted, a choice of signs corresponds to 
an orientation of the Dynkin diagram. Since one can pick an orientation of the Dynkin diagram which is invariant under  $\Aut(\Pi)$,  the group of automorphisms of 
 $\Pi$ can be lifted to a group of automorphisms of  $G$. 
Thus, we have a semi-direct product
\[  \tilde{G} = G \rtimes \Aut(\Pi) = G\rtimes S_3, \] 
where the action of $S_3$ permutes the root subgroups $U_{\alpha}$ and the isomorphisms $x_{\alpha}$. 
\vskip 5pt

Since the outer automorphism group $S_3$ of $G$ is also the automorphism group of the split \'stale cubic $F$-algebra $F^3$, we see that
every cubic \'etale algebra $E$ defines a simply-connected quasi-split form $G_E$ of $G$, whose outer automorphism group is the finite group scheme $S_E$. Thus, 
\[ \tilde G_E = G_E \rtimes S_E \]
is a form of $\tilde{G}$, and it comes equipped with
a pair $B_E \supset T_E$ consisting of a Borel subgroup $B_E$  containing a maximal torus $T_E$, both defined over $F$, as well as a Chevalley-Steinberg system of \'{e}pinglage relative to this pair.

 \vskip 5pt

\subsection{\bf $G_2$ root system.}

The subgroup of $G_E$ fixed pointwise by $S_E$ 
is isomorphic to the split exceptional group of type $G_2$. 
\vskip 10pt

Observe that $B=G_2 \cap B_E$ is a Borel subgroup of $G_2$ and $T = T_E \cap G_2$ is a maximal split torus of $G_2$.   Via the adjoint action of $T$ on $G_E$, we obtain the root system $\Psi_{G_2}$ of $G_2$, so that
\[  \Psi_{G_2} = \Psi|_T. \]
 We denote the short simple root of this $G_2$ root system by $\alpha$ and the long simple root by $\beta$.  Then
\[  \beta = \alpha_0|_T \quad \text{and} \quad \alpha = \alpha_1|_T = \alpha_2|_T = \alpha_3|_T. \]
Thus, the short root spaces have dimension $3$, whereas the long root spaces have dimension $1$. 
For each root $\gamma \in \Psi_{G_2}$, the associated root subgroup $U_{\gamma}$ is defined over $F$ and the Chevalley-Steinberg system of \'{e}pinglage gives isomorphisms:
\[  U_{\gamma} \cong \begin{cases}
\Res_{E/F}\mathbb{G}_a, \text{   if $\gamma$ is short;} \\
\mathbb{G}_a, \text{  if $\gamma$ is long.} \end{cases} \]
 \vskip 10pt

\subsection{\bf The parabolic subgroup $P_E$.}   \label{SS:Heis}

The $G_2$ root system gives rise to 2 parabolic subgroups of $G_E$. One of these is a maximal parabolic $P_E = M_E N_E$ known as the Heisenberg parabolic. Its unipotent radical $N_E$ is a Heisenberg group with center $Z_E = U_{\beta_0}$, see Section 2 in \cite{GGS}. Moreover,
\[  N_E/Z_E  = U_{\beta} \times U_{\beta + \alpha} \times U_{\beta+2\alpha} \times U_{\beta+3\alpha}  
 \cong \mathbb G_a \times \Res_{E/F}\mathbb{G}_a\times \Res_{E/F}\mathbb{G}_a \times \mathbb{G}_a \]
 and 
\[  \tilde{M}_E = M_E \rtimes S_E \cong\GL_2(E)^0  \rtimes S_E  \]
where  
\[ 
\GL_2(E)^0 = \{ g \in \GL_2(E): \det(g) \in F^{\times} \}.
\]

We shall fix the isomorphism $M_E \rtimes S_E \cong\GL_2(E)^0  \rtimes S_E$ as follows.  We first consider the case when $E = F^3$ is split.  
The pinning gives us an identification 
\[ 
M_{\der} (F) \cong \SL_2(F)^3
\] 
such that 
\[  \alpha_1^{\vee}(t)  = \left(  \left( \begin{array}{cc}
t & \\
& t^{-1} \end{array} \right),  1, 1 \right)  \in  \SL_2(F)^3, \]
while  $\alpha_2^{\vee}(t)$ and $\alpha_3^{\vee}(t)$ are defined analogously by cyclically permuting the entries of $\alpha_1^{\vee}(t)$. 
We extend this identification to $M(F)$  by 
\[ 
 \alpha_0^{\vee}(t)  = \left(  \left( \begin{array}{cc}
1 & \\
& t \end{array} \right)  , 
 \left( \begin{array}{cc}
1 & \\
& t \end{array} \right),
\left( \begin{array}{cc}
1 & \\
& t \end{array} \right) \right) \in (\GL_2(F)^3)^0. 
 \]
 Note that, under the identification, 
\[  \beta_0^{\vee}(t)  =  
 \left(  \left( \begin{array}{cc}
t & \\
& t \end{array} \right)  , 
 \left( \begin{array}{cc}
t & \\
& t \end{array} \right),
\left( \begin{array}{cc}
t & \\
& t \end{array} \right) \right)
 \in  (\GL_2(F)^3)^0. \]
 Finally, since the pinning is invariant under the action of $\Aut(\Pi)\cong S_3$, it follows that 
 \[ 
 \tilde{M}(F)\cong( \GL_2(F)^3)^0 \rtimes S_3 
 \] 
 where $S_3$ acts on $( \GL_2(F)^3)^0$ by permuting the components. 
  For a general $E$, one obtains the desired isomorphism by a Galois descent argument.

 \vskip 10pt

 \section{\bf Bhargava's Cube}  \label{S:cube}
 In this section, we shall examine the split case, where the pinning for $G$ gives a $\mathbb Z$-structure on $N/Z$, for more details see Section 4 in \cite{GGS}. 
 
 \vskip 5pt
 
 \subsection{\bf Bhargava's cube.}  \label{SS:cube}  Let $V_2$ be the standard representation of $\SL_2$. 
 Recall that we have identified  $M_{\der}$  with $\SL_2^3$ and $M$ with $(\GL_2^3)^0$. Under this identification, the representation of $M_{\der}$  
 on $N/Z$ is isomorphic to the representation of $\SL_2^3$ on $V=V_2 \otimes V_2 \otimes V_2$.  Since  $\beta_0^{\vee}(t)$ acts on $N/Z$ as multiplication by 
 $t$,  it follows that $( \GL_2^3)^0$ acts on $V$ by the standard action \underline{twisted} by $\det^{-1}$. The group $S_3\cong \Aut(\Pi)$ acts on 
 $V_2 \otimes V_2 \otimes V_2$ by permuting the three factors.  
 
 \vskip 5pt

Since $V$ is an absolutely irreducible $\SL_2^3$-module, the isomorphism of $N/Z$ and $V$ is unique up to a non-zero scalar. Since $\beta_0^{\vee}(t)$ acts on $N/Z$ as multiplication by  $t$, the bijection between $M$-orbits on $N/Z$ and $M$-orbits on $V$ does not depend on the choice on the isomorphism.  If we demand that the isomorphism preserves $\mathbb Z$-structures, i.e. it  gives an isomorphism of $(N/Z)(\mathbb Z)$ and $\mathbb Z^2\otimes \mathbb Z^2\otimes\mathbb Z^2$, then it is unique up to a sign.

 \vskip 5pt 
 An element  $v\in V(F)$ is represented by a cube
 
 
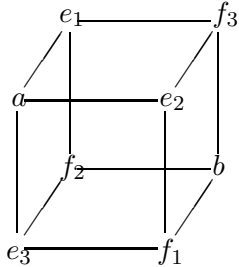
\begin{picture}(100,130)(-130,0) 

\put(18,16){$e_3$} 
\put(25,20){\line(1,0){50}}
\put(75,16){$f_1$}
\put(22,26){\line(0,1){46}}
\put(78,26){\line(0,1){46}}

\put(38,48){$f_2$} 
\put(45,50){\line(1,0){50}}
\put(96,48){$b$}
\put(42,56){\line(0,1){46}}
\put(98,56){\line(0,1){46}}

\put(25,25){\line(2,3){14}}
\put(82,25){\line(2,3){14}}

\put(20,74){$a$}
\put(25,76){\line(1,0){50}}
\put(76,74){$e_2$}

\put(38,106){$e_1$}
\put(45,106){\line(1,0){50}}
\put(96,106){$f_3$}

\put(25,81){\line(2,3){14}}
\put(82,81){\line(2,3){14}}

\end{picture}

\noindent 
where $a, \ldots  , b\in F$, and the vertices correspond to the standard basis in $F^2 \otimes F^2 \otimes F^2$. 
More precisely, we fix this correspondence so that 
\[ 
\left(\begin{array}{c} 1 \\  0 \end{array}\right) \otimes \left(\begin{array}{c} 1 \\  0 \end{array}\right)\otimes 
 \left(\begin{array}{c} 1 \\  0 \end{array}\right) \text{ and } \left(\begin{array}{c} 0 \\  1 \end{array}\right)\otimes  
 \left(\begin{array}{c} 0 \\  1 \end{array}\right) \otimes \left(\begin{array}{c} 0 \\  1 \end{array}\right) 
 \] 
correspond to the vertices marked with letters $a$ and $b$, respectively.  We note that elementary matrices in 
 $\SL_2(F)^3$ act on the space of cubes by the following three types of ``row-column" operations on cubes: 
\begin{itemize}
\item add or subtract the front face from the rear face of the cube, and vice-versa. 
\item add or subtract the top face from the bottom face of the cube, and vice-versa.
\item add or subtract the right face from the left face of the cube, and vice-versa. 
\end{itemize} 

\noindent 
The group $S_3\cong \Aut(\Pi)$ acts as the group of symmetries of the cube fixing these two vertices. We shall often write the cube as 
a quadruple 
\[ 
(a,e,f,b)
\] 
where $e=(e_1,e_2,e_3)$ and $f=(f_1,f_2,f_3)\in F^3$.  
 \vskip 5pt

\subsection{\bf Reduced and distinguished cube.}
It is not hard to see that, using the action of $M(F)$,  every cube can be 
transformed into a cube of the form $(1,0, f, b)$:

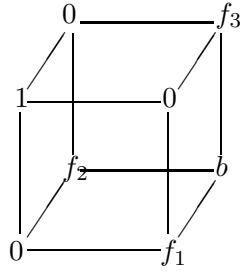
 \begin{picture}(100,130)(-130,0) 

\put(18,16){$0$} 
\put(25,20){\line(1,0){50}}
\put(75,16){$f_1$}
\put(22,26){\line(0,1){46}}
\put(78,26){\line(0,1){46}}

\put(38,48){$f_2$} 
\put(45,50){\line(1,0){50}}
\put(96,48){$b$}
\put(42,56){\line(0,1){46}}
\put(98,56){\line(0,1){46}}

\put(25,25){\line(2,3){14}}
\put(82,25){\line(2,3){14}}

\put(20,74){$1$}
\put(25,76){\line(1,0){50}}
\put(76,74){$0$}

\put(38,106){$0$}
\put(45,106){\line(1,0){50}}
\put(96,106){$f_3$}

\put(25,81){\line(2,3){14}}
\put(82,81){\line(2,3){14}}

\end{picture}

\vskip 5pt

\noindent We shall call such a cube a {\em reduced cube}. In particular, we call the cube 
$v_0 = (1,0,0,-1)$ the {\em distinguished cube}.

\vskip 10pt

 \subsection{\bf Stabilizer of distinguished cube.}
 Let $\Stab_{M}(v_{0})$ and $\Stab_{\tilde M}(v_{0})$ be the respective stabilizers in $M$  and $\tilde M$ of the distinguished cube $v_0 \in V$.   
Since $\Aut(\Pi)$ stabilizes $v_0$,  the group $\Stab_{\tilde M}(v_{0})$  is a semi direct product of $\Stab_{M}(v_{0})$ and $\Aut(\Pi)$. 
We shall now compute $\Stab_{M}(v_{0})$. Let $g=(g_1,g_2,g_3)\in M(F)$ where 
\[ 
g_i= \left(\begin{array}{cc} a_i & b_i \\ c_i & d_i \end{array}\right). 
\] 
Since 
\[ 
v_0= \left(\begin{array}{c} 1 \\  0 \end{array}\right) \otimes \left(\begin{array}{c} 1 \\  0 \end{array}\right)\otimes 
 \left(\begin{array}{c} 1 \\  0 \end{array}\right) -  \left(\begin{array}{c} 0 \\  1 \end{array}\right)\otimes  
 \left(\begin{array}{c} 0 \\  1 \end{array}\right) \otimes \left(\begin{array}{c} 0 \\  1 \end{array}\right)
 \] 
 and 
 \[ 
g\cdot v_0= \det(g)^{-1} \cdot \left(\begin{array}{c} a_1 \\  c_1 \end{array}\right) \otimes \left(\begin{array}{c} a_2 \\  c_2 \end{array}\right)\otimes 
 \left(\begin{array}{c} a_3 \\  c_3 \end{array}\right) -  \det(g)^{-1} \cdot \left(\begin{array}{c} b_1 \\  d_1 \end{array}\right)\otimes  
 \left(\begin{array}{c} b_2 \\  d_2 \end{array}\right) \otimes \left(\begin{array}{c} b_3 \\  d_3 \end{array}\right)
 \] 
 $g\cdot v_0=v_0$  if and only if eight equations hold. Six of these equations are homogeneous. They are: 
 \[ 
 a_1 c_2 a_3 =  b_1 d_2 b_3
 \] 
 \[ 
 a_1 c_2 c_3 = b_1 d_2 d_3
 \] 
 with the additional four obtained by cyclically permuting the indices. 
If we multiply the first equation by $d_3$, the second by $b_3$, and subtract them, then 
\[ 
0= a_1 c_2 a_3 d_3 - a_1 c_2 c_3 b_3= a_1 c_2 (a_3 d_3 - c_3 b_3). 
\] 
Since $a_3 d_3 - c_3 b_3\neq 0$, we have $a_1 c_2=0$. A similar manipulation of these two equations gives $b_1 d_2=0$. 
By permuting the indices,  we have $a_i c_j=b_id_j=0$ for all $i\neq j$. This implies that all $g_i$ are simultaneously diagonal or off diagonal. 
Now it is easy to see that the remaining two equations imply that  $\Stab_M(v_0)$ has two connected components, and the identity component consists of $g=(g_1, g_2, g_3)$ such 
that $g_i$ are diagonal matrices,  $a_i d_i=1$, and $a_1 a_2 a_3=1$.  The other component of $\Stab_M(v_0)$   
contains an element $w=(w_1,w_2,w_3)$ of order 2, where 
\[ 
w_i= \left(\begin{array}{cc} 0 & 1 \\ 1 & 0 \end{array}\right). 
\] 
 We now have a complete description of $\Stab_M(v_0)$ (and of $\Stab_{\tilde M}(v_0)$): 
\[  \Stab_M(v_0) \cong \{ (a_1, a_2, a_3) \in \mathbb{G}_m^3: a_1a_2a_3 =1\} \rtimes \Z/2\Z \cong \mathbb{G}_m^2 \rtimes \Z/2\Z. \]
In particular, we have shown: 
 \vskip 5pt

 \begin{prop}  \label{P:stab-1}
 The stabilizer $\Stab_{\tilde M}(v_0)$ in $\tilde M$ of the distinguished cube $v_0= (1,0,0,-1)$ is 
  isomorphic to the  group of $F$-automorphisms of the  split twisted composition algebra $C_0$.
Indeed, they give identical subgroups of $(\GL_2(F)^3)^0 \rtimes S_3$ where we fix the isomorphism $M(F) \cong (\GL_2(F)^3)^0$ as above.  
 \end{prop} 

\vskip 10pt

  \subsection{\bf Three quadratic forms.}
One key observation in \cite{B1} is that one can slice the cube (given in the picture in \S \ref{SS:cube})  in three different ways, giving three pairs of matrices: 

$$ 
A_1 =\left(\begin{array}{cc} a & e_2 \\ e_3 & f_1 \end{array}\right) \hskip 5pt 
B_1 =\left(\begin{array}{cc} e_1 & f_3 \\ f_2 & b \end{array}\right)  
$$ 
$$ 
A_2 =\left(\begin{array}{cc} a & e_3 \\ e_1 & f_2 \end{array}\right) \hskip 5pt 
B_2 =\left(\begin{array}{cc} e_2 & f_1 \\ f_3 & b \end{array}\right)  
$$ 
$$ 
A_3 =\left(\begin{array}{cc} a & e_1 \\ e_2 & f_3 \end{array}\right) \hskip 5pt 
B_3 =\left(\begin{array}{cc} e_3 & f_2 \\ f_1 & b \end{array}\right)  
$$ 
Note that the pairs $(A_{2},B_{2})$ and $(A_{3},B_{3})$ are obtained by rotating the 
pair $(A_{1},B_{1})$  about the axis passing through $a$ and $b$. 
For each pair  $(A_i, B_i)$, Bhargava defines a quadratic binary form by 
\[  Q_i=-\det(A_i x + B_i y). \] 
\smallskip 

\begin{prop} 
Given a cube $v$, the three forms  $Q_{1}$, $Q_{2}$ and $Q_{3}$ 
have the same discriminant $\Delta=\Delta(v)$. 
\end{prop}

\begin{proof}  We may assume the cube  is reduced. Now an easy computation show that the three forms  are 
$$ 
\begin{cases} Q_1(x,y) = -f_1x^2-bxy+f_2 f_3y^2 \\
              Q_2(x,y) = -f_2 x^2 - bxy +f_3 f_1 y^2 \\
              Q_3(x,y) = -f_3x^2 -bxy +f_1 f_2y^2. 
\end{cases}
$$
These forms have the same discriminant $\Delta=b^2+4f_1f_2f_3$. 
\end{proof}

\vskip 10pt

\subsection{\bf Quartic invariant.}
To every cube $v \in V$, the discriminant $\Delta(v)$ described in the previous proposition is a homogeneous quartic polynomial in $v$, which is invariant under the action of $\SL_2(F)^3$. This describes the quartic invariant of the prehomogeneous vector space $V$. An explicit computation gives the following formula:
\begin{align}
\Delta=\, \,& a^2 b^2   -2ab (e_1 f_1 + e_2f_2 +  e_3 f_3) + e_1^2 f_1^2 + e_2^2 f_2^2 +  e_3^2 f_3^2
\notag \\
&+ 4a f_1 f_2 f_3 + 4 b e_1 e_2 e_3 - 2( e_1e_2f_1f_2 + e_2e_3f_2f_3 +  e_3e_1f_3f_1 ).
\notag \end{align}
If $v$ is reduced, then this simplifies to $\Delta(v)  = b^2 + 4 f_1f_2f_3$. It is easy to check that for $g \in M$, one has
\[  \Delta(g \cdot v)  = \det (g)^2 \cdot \Delta(v). \]
 Thus, we see that $\Delta$ gives a well-defined map
 \[ \Delta:  \{ \text{generic $\tilde{M}(F)$-orbits on $V(F)$}\} \longrightarrow F^{\times}/ F^{\times 2} =
 \{ \text{\'etale quadratic $F$-algebras} \}. \]
\vskip 10pt

 \section{\bf $E$-twisted Bhargava Cube}
 
 Now we can extend the discussion of the previous section to the case of general $E$, where $V_E = F \oplus E \oplus E \oplus F$ 
 and $\tilde{M}_E= \GL_2(E)^0 \rtimes S_E$, via a Galois descent using a cocycle in the class of
 \[  [E] \in H^1(F, \Aut(\Pi))  =H^1(F, S_3). \]
 A cube is a quadruple $v=(a,e,f,b)$, where $e,f\in E$. As in the split case, we shall call the cubes of the form $v = (1,0,f,b)$ reduced,  and the vector $v_{0,E} = (1,0,0,-1)$ the $E$-distinguished cube.
 \vskip 5pt
 
\subsection{\bf Quartic invariant.} 
By Galois descent, we see that  the basic polynomial invariant $\Delta_E$   is given by
 \[  \Delta_E(a,e,f,b) =  a^2b^2 -2ab Tr_{E/F}(ef)+ Tr_{E/F}(e^2f^2) + 4a N_{E/F}(f) + 4b N_{E/F}(e) -2 Tr_{E/F}(e^{\#} f^{\#}) \]
 If $v$ is reduced, then this simplifies to:  
  \[  \Delta_E(1,0,f,b) = b^2 + 4 \cdot N_{E/F}(f). \] 
 \vskip 5pt
 
 \subsection{\bf Group action.}  \label{SS:action} 
 It is useful to note the action of certain elements of $\GL_2(E)^0$ on $V_E$. Specifically,
 $\sigma \in S_E$ acts by $\sigma(a,e,f,b) = (a, \sigma(e), \sigma(f), b)$.  Moreover,
  the diagonal torus elements
 \[  t_{\alpha,\beta} = \left( \begin{array}{cc}
 \alpha & 0 \\
 0 & \beta \end{array} \right)  \quad \text{with $\alpha \beta \in F^{\times}$}\]
 acts by
 \[  (a,e,f,b) \mapsto (\alpha^{\#}\beta^{-1} a ,  \alpha^{\#} \alpha^{-1}  e, \beta^{\#} \beta^{-1} f, \beta^{\#} \alpha^{-1} b). \]
 It is easy to check that
 \[  \Delta_E( t_{\alpha, \beta} \cdot v)  = (\alpha\beta)^2 \cdot \Delta_E(v). \]
 Since the action of $\SL_2(E)$ and $S_E$ preserve $\Delta_E$, we see that 
 \[  \Delta_E( g \cdot v)  = (\det g)^2 \cdot \Delta_E(v) \]
 so that  $\Delta_E$ induces a map
 \[  \{ \text{$\tilde{M}_E$-orbits on $V_E$} \} \longrightarrow F^{\times}/ F^{\times 2} = \{ \text{\'etale quadratic algebras} \}. \]
 In addition, the standard Weyl group element 
 \[  w  =\left( \begin{array}{cc}
 0 & 1 \\
 1 & 0 \end{array} \right) \in \GL_2(E)^0 \]
 acts on
 \[  w :  (a,e,f,b) \mapsto (-b,-f,-e,-a). \]
 
 \vskip 5pt

 \vskip 5pt

\subsection{\bf Stabilizer of distinguished $E$-cube.}  \label{SS:stab-E}
We can readily determine the stabiliser of the $E$-distinguished cube . Namely, under the action described in \S \ref{SS:action},  it is easy to see that the subgroup
\[  E^1  = \{  \left( \begin{array}{cc}
\alpha & \\
  & \alpha^{-1} \end{array}  \right): \alpha \in E^1 \} \subset \SL_2(E) \]
  fixes the $E$-distinguished cube $v_{0,E}$. So does the Weyl group element $w$. Thus we see that 
  \[  \Stab_{M_E}(v_{0,E}) \cong E^1 \rtimes \Z/2\Z \quad \text{and} \quad \Stab_{\tilde{M}_E}(v_{0,E} ) =  
  E^1 \rtimes (\Z/2\Z \times S_E). \] 
In particular, we have shown:
\vskip 5pt

 \begin{prop} The stabilizer $\Stab_{\tilde{M}_E}(v_{0,E} )$ in $\tilde M_E$ of the $E$-distinguished cube $(1,0,0,-1)$ is 
  isomorphic to the  group of $F$-automorphisms of the twisted composition algebra $C_E$ introduced in \S \ref{SS:example}. Indeed, they are identical as subgroups of $\GL_2(E)^0 \rtimes S_E(F)$ under the fixed isomorphism $M_E(F) \cong \GL_2(E)^0$. 
 \end{prop} 

\vskip 10pt

\section{\bf Generic Orbits}
We come now to the main result of this paper: the determination of the generic $\tilde{M}_E(F)$-orbits in $V_E(F)$. 
\vskip 5pt

\subsection{\bf A commutative diagram.}
We have the following commutative diagram 
\begin{equation} \label{E:comm}
\xymatrix{ H^1(F, \Stab_{\tilde{M}}(v_{0} )) \ar[r] \ar[d] & H^1(F, \tilde{M}) \ar[d] \\ 
H^1(F, \Aut_F(C_0,Q_0,\beta_0)) \ar[r] & H^1(F, S_3)}
\end{equation}

\vskip 5pt

We make several observations about this commutative diagram.
\vskip 5pt

\begin{lemma}
(i) The   first vertical arrow is bijective.
\vskip 5pt

\noindent (ii) The second vertical arrow is bijective.

\vskip 5pt
\noindent (iii) The horizontal arrows are surjective.
\end{lemma}

 \begin{proof} 
\noindent (i) This follows by Proposition \ref{P:stab-1}.
\vskip 5pt

\noindent (ii) Let the second vertical arrow be denoted by $\psi$. Since $\tilde{M}$ is a semi-direct product of $M$ and $S_3$, the map $\psi$  is surjective. For injectivity, we shall use 
the exact sequence of pointed sets:
\[ 
1 \rightarrow H^1(F, M) \rightarrow H^1(F, \tilde{M})\rightarrow H^1(F, S_3)  \rightarrow 1. 
\] 
Let $c\in H^1(F, S_3)$ and let $E$ be the \'etale cubic algebra corresponding to $c$. Then 
$M_E$ is the twist of $M$ by $c$. In order to prove that $\psi^{-1}(c)$ consists of one element, 
 it suffices to show that $H^1(F,M_E)$ is trivial, by the twisting argument on page 50 in \cite{Se}.  We have an exact sequence of algebraic groups 
 \[ 
 1 \rightarrow M_{E,\der}  \rightarrow M_E \rightarrow \GL_1  \rightarrow 1, 
\] 
where $M_{E,\der} \cong \Res_{E/F} \SL_2$. By Hilbert's Theorem 90,  $H^1(F, \GL_1)$ is trivial.  Since 
\[ 
H^1(F, \Res_{E/F} \SL_2) = H^1(E, \SL_2) = 0
\] 
(see \cite{Se}, page 130),  it follows that $H^1(F, M_{E})$ is trivial. 
 \vskip 5pt

\noindent (iii) This follows because $\Stab_{\tilde{M}}(v_{0} )  = \Stab_{M}(v_{0} ) \rtimes \Aut(\Pi)$, hence 
\[ 
H^1(F, \Stab_{\tilde{M}}(v_{0} )) \rightarrow H^1(F, \Aut(\Pi))
\] 
 has a natural splitting.
\end{proof} 

\vskip 5pt

\subsection{\bf Determination of orbits.}
We can now determine the generic $\tilde{M}_E(F)$-orbits on $V_E(F)$. 
\vskip 5pt

\begin{thm}  \label{T:main-orbit}
Fix an \'etale cubic $F$-algebra $E$.
\vskip 5pt

\noindent (i) The generic $\tilde{M}_E(F)$-orbits on $V_E(F)$ are in bijective correspondence with the set of $F$-isomorphism classes of $E$-twisted composition algebras over $F$, with the orbit of $v_{0,E} = (1,0,0,1)$ corresponding to the twisted composition algebra $C_E$ introduced in \ref{SS:example}.
\vskip 5pt

\noindent (ii) The generic $M_E(F)$-orbits on $V_E(F)$ are in bijective correspondence with the set of $E$-isomorphism classes of $E$-twisted composition algebras over $F$.
\vskip 5pt
\noindent  (iii) There is a commutative diagram
\[  \begin{CD}
\{ \text{$E$-twisted composition algebras} \}  @>>> \{ \text{\'etale quadratic $F$-algebras} \} \\
@VVV   @VVV   \\
\{ \text{generic $\tilde{M}_E$-orbits on $V_E$} \}  @> >>   F^{\times}/ F^{\times 2} \end{CD} \] 
where the bottom arrow is the map induced by $\Delta_E$ (see \S \ref{SS:action}).
\end{thm}
\vskip 5pt

\begin{proof}

(i) Given a cohomology class $[E] \in H^1(F, S_3)$ corresponding to an \'etale cubic $F$-algebra,  we consider the fibers of the two horizontal arrows in the commutative diagram (\ref{E:comm}) over $[E]$. 
Since the map $\Stab_{\tilde{M}}(v_{0} ) \longrightarrow S_3$ splits, the fiber of the second horizontal arrow has a distinguished element which corresponds to the twisted composition algebra $C_E$. 
Similarly, the fiber over $[E]$ of the  first horizontal arrow has a distinguished point which corresponds to the orbit of $v_{0,E} = (1,0,0,-1)$. Moreover, these two distinguished point correspond under the first vertical arrow. 
\vskip 5pt

By the twisting argument, see \cite{Se}, page 50,  we see that both fibers in question are naturally identified  with
\[  {\rm Ker}(H^1(F, \Stab_{\tilde{M}_E}(v_{0,E})) \longrightarrow H^1(F, \tilde{M}_E)). \]
  Thus, the fiber of the first horizontal map over $[E]$ are 
the generic $\tilde M_E$-orbits in $V_E$, while the fibers of the second map are $F$-isomorphism classes of $E$-twisted composition algebras. 
\vskip 5pt

\noindent (ii) The bijection follows because both sets are in natural bijection with
$H^1(F, \Stab_{M_E}(v_{0,E}) =  H^1(F, \Aut_E(C_E))$.

\vskip 5pt

\noindent (iii) Suppose an $E$-twisted composition algebra is represented by a cocycle 
\[ 
(a_{\sigma}) \in H^1(F, \Stab_{M_E}(v_{0,E}) ).
\] 
 Then the associated \'etale quadratic $F$-algebra $K$ corresponds to the group homomorphism
\[  \begin{CD} \eta_K:  \Gal(\bar{F}/F) @>>> \Stab_{M_E}(v_{0,E}) (\bar{F})  @>>> \Z/2\Z  \end{CD} \]
given by $\sigma \mapsto a_{\sigma} \mapsto  \pi(a_{\sigma})$, where $\pi: \Stab_{M_E}(v_{0,E}) \rightarrow \Z/2\Z$ is the natural projection. In fact, regarding $\Stab_{M_E}(v_{0,E})  \subseteq M_E$ as described in \S \ref{SS:stab-E}, we see that the map $\pi$ is simply given by the determinant map on $M_E = \GL_2(E)^0$. 
\vskip 5pt

On the other hand, the cocycle splits in $H^1(F, M_E) =0$, so that we may write
\[  a_{\sigma}  =  g^{-1} \cdot \sigma(g) \quad \text{for some $g \in M_E(\bar{F})$}. \]
Then the $\tilde{M}_E$-orbit associated to $(a_{\sigma})$ is that of $g \cdot v_{0,E}$. 
Now we have:
\[  \Delta_E(g \cdot v_{0,E}) = \det(g)^2 \cdot \Delta_E(v_{0,E}) = \det(g)^2 \]
and
\[  \eta_K(\sigma)  = \det(a_{\sigma}) = \det(g)^{-1} \cdot \sigma( \det(g))  \]
for any $\sigma \in \Gal(\bar{F}/F)$. This shows that $\det(g)$ is a trace zero element in $K$, so that $K$ is represented by the square class of $\det(g)^2 \in F^{\times}$, as desired.
\vskip 5pt

The theorem is proved.
\end{proof}
\vskip 5pt

In particular, we have established Theorem \ref{T:main}.   However, the bijection  between the generic $\tilde{M}_E(F)$-orbits on $V_E(F)$ and the $E$-isomorphism classes of twisted composition algebras is obtained by a Galois cohomological argument, which is quite formal and  not at all explicit. For applications, it is necessary to have an explicit description of the bijection. We shall arrive at such an explicit description in the following sections.

\vskip 10pt

\section{\bf Re-interpreting Bhargava}  \label{S:reinter}

In this section, revisiting the case when $E = F^3$ is split, we shall re-interpret Bhargava's results in \cite{B1} in the framework of twisted composition algebras,  leading to an explicit recipe for the bijection  in Theorem \ref{T:main-orbit}.  
\vskip 15pt
 
\subsection{\bf Bhargava's result.}
We first review briefly Barghava's results and, following him, we shall work over $\mathbb Z$. Note that we have an action of the group 
$\SL_2(\mathbb Z)^3$ on the set of integer valued cubes, by the ``row-column" operations  as described in \S \ref{SS:cube}.
\vskip 5pt

In order to  state the main result of Bhargava, we need a couple of definitions.
Fix a discriminant $\Delta$. Let $K=\mathbb Q(\sqrt{\Delta})$ and $R$ the unique 
order of discriminant $\Delta$.  A module $M$ is a full lattice in $K$. In particular, it is 
a $\mathbb Z$-module of rank 2. We shall write $M=\{u,v\}$ if  $u$ and $v$ span $M$.  For example, 
\[ 
R=\{ 1, \frac{\Delta +\sqrt{\Delta}}{2}\} .
\] 
By fixing this basis of $R$, we have also fixed a preferred orientation of bases of modules. 
An oriented module is a pair $(M, \epsilon)$ where $\epsilon$ is a sign. If $M=\{u,v\}$, then $M$ becomes an oriented module
$(M,\epsilon)$,  where  $\epsilon=1$ if and only if  the orientation of $\{u,v\}$ is preferred. The norm of an oriented module 
$(M,\epsilon)$ is $N(M)= \epsilon\cdot [R:M]$. 
\vskip 5pt

Then:
\vskip 5pt

\begin{itemize}
\item  a triple of oriented modules $(M_1, M_2, M_3)$ with $R$ as the 
multiplier ring,  is said to 
be {\em colinear}, if there exists $\delta\in K^{\times}$ such that the product of the 
three oriented modules is a principal oriented ideal $((\delta), \epsilon)$
where $\epsilon=sign(N(\delta))$, i.e.,  $M_1M_2 M_3 = (\delta)$, as 
ordinary modules, and $N(M_1)N(M_2)N(M_3)=N(\delta)$. 
\vskip 5pt

\item a cube is {\em projective} of discriminant $\Delta$
 if the three associated forms are primitive and have the discriminant $\Delta$.

\vskip 5pt

\item  two triples of oriented modules $(M_{1}, M_{2}, M_{3})$ and 
$(M'_{1}, M'_{2}, M'_{3})$ are equivalent if there exist $\mu_{1}, \mu_{2}, \mu_{3}$ 
in $K^{\times}$ such that $M'_{i}=\mu_{i} M_{i}$ and 
$\epsilon'_{i}= sign(N(\mu_{i}))\epsilon_{i}$ for $i=1,2,3$. 
\end{itemize}
\vskip 5pt

Then Bhargava \cite{B1} showed:
\vskip 5pt

\begin{thm}  \label{T:bhargava}
There is a bijection, to be described in the proof, between the equivalence classes
 of oriented colinear triples of discriminant $\Delta$ and   
$SL_2(\mathbb Z)^3$-equivalence classes of projective cubes of discriminant $\Delta$. 
\end{thm} 

\noindent{\bf \underline{Sketch of Proof:}}
 Let $v$ be a projective cube. Again, without any loss of generality we can assume that 
the cube is reduced and that the numbers $f_1$, $f_2$ and $f_3$ are nonzero. Define three modules by   
$$ 
M_1 = \{ 1, \frac{b-\sqrt{\Delta}}{2f_1} \}, \hskip 2pt   
M_2 =\{ 1, \frac{b-\sqrt{\Delta}}{2f_2}\} \hskip 2pt \text{and }
M_3=\{ 1, \frac{b-\sqrt{\Delta}}{2f_3} \}. 
$$
The norms of the three modules are $-1/f_1, -1/f_2$ and $-1/f_3$, respectively, if we take the 
given bases to be proper. For $\delta$ we shall take  
$$ 
\delta = -\frac{2}{b+\sqrt{\Delta}}, 
$$
which has the correct norm $-1/f_1 f_2 f_3$. 

The modules  $M_i$, with given oriented bases, correspond to the quadratic forms $Q_i$. More precisely, if 
\[ 
z_i= x_i + y_i\frac{b-\sqrt{\Delta}}{2f_i}\in M_i
\] 
then 
\[ 
-f_iN(z_i) =Q_i(x_i,y_i)=-f_i x^2_i -  bx_i y_i + f^{\#}_i y_i^2 
\] 
where $f^{\#}=(f_2 f_3, f_3 f_1, f_1 f_2)$. 
\qed

\vskip 10pt

\subsection{\bf Integral Twisted Composition algebras}  \label{SS:integral} 

We can now give a re-interpretation of Bhargava's results, in particular of Barghava's triples $(M_1,M_2,M_3)$, in the framework of twisted composition algebras. 
 Assume the notation from the previous subsection, so that  $M_1 M_2 M_3= (\delta)$. Set 
\[  C =M_1 \oplus M_2\oplus M_3. \]
We shall define a pair of tensors $(Q, \beta)$ on $C$ as follows:
\vskip 5pt

\begin{itemize}
\item Define a quadratic form  $Q: C \rightarrow \mathbb Z \times \mathbb Z\times \mathbb Z$   by 
\[ 
Q(z_1,z_2,z_3)= (-f_1 N(z_1), -f_2 N(z_2), -f_3 N(z_3))= -f\cdot (N(z_1), N(z_2), N(z_3)). 
\]
\vskip 5pt

\item Define a quadratic map   $\beta: C \rightarrow C$ by
\[ 
\beta(z_1,z_2,z_3) = \delta(f_2 f_3 \bar z_2 \bar z_3, f_3 f_1 \bar z_3 \bar z_1,f_1 f_2 \bar z_1 \bar z_2)=
\delta \cdot f^{\#} \cdot (\bar z_1, \bar z_2,\bar z_3)^{\#}. 
\]
\end{itemize}
The relations  $M_1 M_2 M_3 = (\delta)$ and $M\bar M=N(M)$ imply that $\beta$ is well defined. Moreover, 
using $N(\delta)=-1/f_1 f_2 f_3$, one checks that 
\[ 
Q(\beta(z_1,z_2,z_3))=Q(z_1,z_2,z_3)^{\#}
\] 
and 
\[ 
N_C(z_1,z_2,z_3)= Tr (\frac{z_1 z_2 z_3}{\delta}). 
\] 
Thus  the triple $(C,Q,\beta)$ is a twisted composition algebra over $\mathbb{Z}$. 
\vskip 5pt

In terms of  the coordinates $(x_i,y_i)$ given by 
\[ 
z_i= x_i + y_i\frac{b-\sqrt{\Delta}}{2f_i},
\] 
we have seen in the sketch  proof of Theorem \ref{T:bhargava} that
\[  Q_i(z_i)=  -f_iN(z_i) =-f_i x^2_i - bx_i y_i + f^{\#}_i y_i^2. \]
We shall now do the same for $\beta$. Write $\beta(z_1,z_2,z_3)=(z'_1,z'_2,z'_3)$, and  let $(x'_i,y'_i)$ be the coordinates of $z'_i$.  A short calculation shows that 
\[ 
x'_1 = - \left(\begin{array}{cc} x_3 &  y_3 \end{array}\right) \left(\begin{array}{cc} 0 & f_3 \\ f_2 & b \end{array}\right) 
\left(\begin{array}{c} x_2 \\
  y_2 \end{array}\right) 
\] 
\[ 
y'_1 =  \left(\begin{array}{cc} x_3 &  y_3 \end{array}\right) \left(\begin{array}{cc} 1 & 0 \\ 0  & f_1 \end{array}\right) 
\left(\begin{array}{c} x_2 \\
  y_2 \end{array}\right)
\] 
while the expressions for $(x'_2,y'_2)$ and $(x_3,y_3)$ are obtained by cyclicly permuting the indices.  
\vskip 5pt

There are two important observations to be made here:
\vskip 5pt

\begin{itemize}
\item Firstly, these formulas make sense for any triple $(f_1,f_2,f_3)$ and any $b$, i.e. the $f_i$'s can be zero. 
The axioms of twisted composition algebra are satisfied for formal reasons. 
For example,  if $(f_1,f_2,f_3)=(0,0,0)$ and $b=-1$, we get the split algebra $C_0$.  
\vskip 5pt

\item Secondly, the two matrices are two opposite faces of the cube. This gives a hint how to directly associate a composition algebra to any cube in general (i.e. not just a reduced cube). 
\end{itemize}
\vskip 10pt

\subsection{\bf From cubes to twisted composition algebras.}
The above discussion suggests an explicit recipe for
 associating a twisted composition algebra over $F \times F\times F$ to any cube $v \in V(F)$. 
\vskip 5pt

Let $C =F^2 \times F^2 \times F^2$. An element  $z\in C$ is 
a triple $(z_1,z_2,z_3)$ of column vectors  $z_i= \left(\begin{smallmatrix}x_i\\
y_i \end{smallmatrix}\right)$. Slice a cube into three pairs of $2\times 2$-matrices $(A_i,B_i)$, as before and let
\[  Q_i(z_i)=-\det(A_i x_i + B_i y_i).\]
Then we set:
\vskip 5pt

\begin{itemize}
\item  $Q: C  \rightarrow F \times F\times F$ is defined by 
\[ 
Q(z_1,z_2,z_3)= (Q_1(z_1), Q_2(z_2), Q_3(z_3)). 
\] 
\vskip 5pt

\item $\beta: C \rightarrow C$ is defined by
\[   \beta(z_1,z_2,z_3)=(z'_1,z'_2,z'_3) \]
where $z'_i= \left(\begin{smallmatrix}x'_i\\
y'_i \end{smallmatrix}\right)$, 
\[ 
x'_1=  -z_3^{\top} B_1 z_2 , \hskip 10pt  
x'_2=  -z_1^{\top} B_2 z_3 , \hskip 10pt 
x'_3= - z_2^{\top} B_3 z_1 
\] 
and 
\[ 
y'_1=   z_3^{\top} A_1 z_2  , \hskip 10pt  
y'_2=  z_1^{\top} A_2 z_3 , \hskip 10pt 
y'_3=   z_2^{\top} A_3 z_1. 
\] 
\end{itemize}
\vskip 5pt

Thus, starting from a cube $v$, we have defined a pair of tensors $(Q, \beta)$ on  $C=F^2 \times F^2 \times F^2$.  Let 
\[ \tilde{\phi}:  V(F) \longrightarrow \{  \text{tensors $(Q, \beta)$ on $C$} \}  \]
be the resulting map.  We may express this map using the coordinates $(a,e,f,b)$ of a cube. A short calculation gives:
\begin{align}  
&Q(x,y)  = (e^{\#} - af) x^2 + (- ab - 2ef + Tr(ef)) xy + (f^{\#} - be)  y^2   \notag \\ 
&\beta(x,y)  = (-ex^{\#} - b y^{\#} - (fx) \times y,   ax^{\#} + fy^{\#} + (ey) \times x).  \notag 
\end{align}

In the next section, we shall study the properties of the map $\tilde{\phi}$; for example, we shall show that a $(Q,\beta)$ in the image of $\tilde{\phi}$ does define a twisted composition algebra on $C$.

\vskip 10pt

\section{\bf Explicit Parametrization}  \label{S:explicit}

Using the results of the previous section, we can now give an explicit description of the bijection between  $\tilde{M}_E(F)$-orbits of nondegenerate cubes and $F$-isomorphism classes of $E$-twisted composition algebras. 
 \vskip 5pt
 
\subsection{\bf Definition of $\tilde{\phi}$.}
Let us  write $C= E\cdot e_1  \oplus E\cdot e_2$.   Motivated by the case where $E = F^3$ studed in the previous section, we define the map 
\[ \tilde{\phi}:  V_E(F) \longrightarrow \{  \text{tensors $(Q, \beta)$ on $C$} \}  \]
 using the coordinates $v = (a, e, f, b)$ of a cube, with $a,b \in F$ and $e,f \in E$ by: 
\begin{align} \label{E:Qb} 
&Q(x,y)  = (e^{\#} - af) x^2 + (-ab-2ef + Tr(ef)) xy + (f^{\#} - be)  y^2  \\
&\beta(x,y)  = (-ex^{\#} - b y^{\#} - (fx) \times y,   ax^{\#} + fy^{\#} + (ey) \times x).  \notag 
\end{align}
In particular, for a reduced cube $(1,0,f,b)$, one has
\begin{align} \label{E:Qb2} 
&Q(x,y)  = -f  x^2 - b xy + f^{\#}   y^2  \\
&\beta(x,y)  = ( -b y^{\#} - (fx) \times y,   x^{\#} + fy^{\#}).  \notag 
\end{align}
Thus, the image of the distinguished cube $v_{E,0} = (1,0,0,-1)$ is the algebra $C_E$. 
Observe also that one has
\begin{equation}
\label{E:reduced-form}
 \beta(1,0)  = (0,1) \quad \text{and}\quad \beta(0,1) = (-b, f). \end{equation}
Thus, the standard basis $\{e_1, e_2 \}$ is a reduced basis with respect to $(Q,\beta)$, in the sense of Section \ref{S:reduced-basis}. 
\vskip 10pt

One has:
\vskip 5pt

\begin{prop}  \label{P:bij}
(i) The map $\tilde{\phi}$ is injective.
\vskip 5pt

(ii) For $g \in \GL_2(E)^0$ and for $\sigma \in S_E(F)$, , one has
\[  \tilde{\phi}(g \cdot v)  = {^t}g^{-1} \cdot \tilde{\phi}(v) \quad \text{and} \quad   \tilde{\phi}(\sigma \cdot v) =  \sigma \cdot \tilde{\phi}(v) \]
for any $v \in V_E(F)$.
\vskip 5pt

Thus, the map $\tilde{\phi}$ is $\GL_2(E)^0\rtimes S_E$-equivariant, with respect to the outer automorphism 
$(g, \sigma)  \mapsto ( {^t}g^{-1}, \sigma)$ of $\GL_2(E)^0 \rtimes S_E$, and where the action of $\GL_2(E)^0  \rtimes S_E$ on the set of $(Q, \beta)$ is given as in Section \ref{SS:isom}. 
\vskip 5pt

(iii) For any nondegenerate cube $v$, $\tilde{\phi}(v)  = (Q,\beta)$ defines a twisted composition algebra on 
$C$.  
\end{prop}
\vskip 5pt

\begin{proof}
(i) If $\tilde{\phi}(a,e,f,b) = (Q, \beta)$, then 
\[  \beta(1,0) = (-e,a) \quad \text{and} \quad \beta(0,1) = (-b, f). \]
Hence the cube $(a,e,f,b)$ is uniquely determined by $\beta$.
\vskip 5pt

\noindent (ii) We can verify this equivariance property over $\overline{F}$; thus we only need to check it for $E = F^3$.  For the central element $(t,t,t) \in \GL_2(E)^0$ or the element $\sigma \in S_E$, the desired equivariance property is clear. Thus, it remains to verify it for elementary matrices such as 
\[   g = (E_u, 1,1) = \left( \left( \begin{array}{cc}
1 & u \\
0 & 1 \end{array} \right) , 1, 1\right)  \in (\GL_2(F) \times \GL_2(F) \times \GL_2(F))^0. \]
  Now if the cube $v$ has a pair of faces $(A_1, B_1)$, then the corresponding pair for $g \cdot v$ is 
 \[  (A'_1, B'_1) = (A_1 + uB_1, B_1). \]
 Slicing the cube in the other two ways, we obtain
 \[  (A_2', B_2')  = ( E_u A_2 , E_u B_2) \quad \text{and} \quad (A'_3, B'_3)  = (A_3E_u^t, B_3E_u^t). \] 
Hence, if  $\tilde{\phi}(g \cdot v) = (Q',\beta')$, then $\beta'$  is given on $(z_1, z_2, z_3) \in F^2 \times F^2 \times F^2$ by
 \[  
 \left( \begin{array}{ccc}
 x'_1 & x'_2 & x'_3 \\
 y'_1 & y'_2 & y'_3 \end{array} \right)  = \left( \begin{array}{ccc}
   - z_3^t B_1 z_2 &  - z_1^t E_u B_2 z_3 &   -z_2^t B_3 E_u^t  z_1  \\
  z_3^t (A_1 + uB_1) z_2 &  z_1^t E_u A_2   z_3 &   z_2^t  A_3E_u^t z_1 \end{array} \right). \] 
 On the other hand, ${^t}g^{-1} $ acts on $\beta$ by precomposing by  $({^t}g^{-1})^{-1}= g^t$, and post composing by ${^t}g^{-1}$: 
 \begin{align}
  {^t}g^{-1}  \cdot  \beta (g^t(z_1, z_2, z_3))  
 &= {^t}g^{-1} \cdot  \beta(E_u^t z_1, z_2, z_3) \notag \\
  &={^t}g^{-1} \cdot  \left( \begin{array}{ccc}
   - z_3^t B_1 z_2 &  - z_1^t E_u B_2 z_3 &   -z_2^t B_3 E_u^t  z_1  \\
  z_3^t A_1 z_2 &  z_1^t E_u A_2   z_3 &   z_2^t  A_3E_u^t z_1 \end{array} \right)\notag \\
 &= 
  \left(  \begin{array}{ccc} 
 - z_3^t B_1 z_2    &  - z_1^t E_u B_2 z_3 &   -z_2^t B_3 E_u^t  z_1 \\
z_3^t (A_1+uB_1)z_2  & z_1^t E_u A_2   z_3 &   z_2^t  A_3E_u^t z_1  \end{array} \right)   \notag \\
&=\beta'(z_1,z_2,z_3). \notag
\end{align}

\vskip 5pt

\noindent (iii) Again, we may work over $\overline{F}$ and hence we may assume that $E = F^3$. If $v$ is a reduced cube, we  have seen in  \S \ref{SS:integral} that $(Q,\beta)$ defines a twisted composition algebra on $E^2$. Since every $\tilde{M}(F)$-orbit contains a reduced cube, the result follows by (ii).
\end{proof}
 \vskip 5pt
 
 The occurrence of the outer automorphism $g \mapsto {^t}g^{-1}$ is natural here. Indeed, assume that $E = F^3$ and  regard $\GL_2(F)$ as $\GL(V)$ for a $2$-dimensional $F$-vector space $V$. Then the quadratic map $\beta$ is an element of $(V^{\ast})^{\oplus 3} \otimes_F (V^{\ast})^{\oplus 3} \otimes_F V^{\oplus 3}$, whereas its associated cube is an element in $V \otimes_F V \otimes_F V \otimes_F \det(V)^{-1}$.
Thus scaling a cube by $t \in F^{\times}$  corresponds to scaling $\beta$ by $t^{-1}$.    
\vskip 10pt

\subsection{\bf Reduced cubes and bases} To describe the image of $\tilde{\phi}$, we examine the case of reduced cubes more carefully.  

\vskip 5pt

\begin{prop}  \label{P:reduced-basis2}
 Suppose that the pair $(Q,\beta)$ defines a twisted composition algebra structure on $E^2$ such that the standard basis $\{e_1, e_2\}$ is reduced (i.e. $\beta(e_1)  = e_2$). Then $(Q,\beta)$ 
   is the image under $\tilde{\phi}$ of the reduced cube $v = (1,0,-Q(e_1),-N_{Q,\beta}(e_1))$. Moreover, 
   $\Delta_E(v)  =  \Delta_{Q,\beta}(e_1)$ (where the $\Delta$ on the LHS is the quasi-invariant form on the space $V_E$ of cubes while that on the left is defined in Proposition \ref{P:reduced-basis}).
\end{prop}
\vskip 5pt

 \begin{proof}
 We need to show that $Q$ and $\beta$ is uniquely determined by $f = -Q(e_1)$ and $b =- N_{Q,\beta}(e_1)$.
 Since 
 \[  Q(e_2) = Q(\beta(e_1)) = f^{\#} \quad \text{and} \quad b_Q(e_1, e_2) = b_Q(e_1, \beta(e_1)) = N(e_1) = -b, \]
 we see that $Q$ is uniquely determined. Then $\beta(xe_1 + ye_2)$ is uniquely determined by
  (\ref{E:beta-reduced}) in Lemma \ref{L:ID}.
  Finally, observe that
  \[  \Delta_E(v)  =  \Delta_{Q,\beta}(e_1) =  b^2 + 4 N_E(f). \]
 \end{proof} 
 \vskip 5pt

 \subsection{\bf Good bases.}
 We call a basis of $C$ a {\em good basis} if it is in the $\Aut_E(C)^0\cong \GL_2(E)^0$-orbit of a reduced basis. 
By Proposition \ref{P:reduced-basis}(iv), this notion is independent  of the choice of the reduced basis.
Similarly, since the action of $S_E$ preserves the set of reduced cubes, the notion of good bases does not depend on whether one uses $\Aut_E(C)^0$ or $\Aut_F(E,C)^0 \cong \GL_2(E)^0 \rtimes S_E$. 

\vskip 5pt

As a consequence of the proposition, we have:
\vskip 5pt

\begin{cor}  \label{C:reduced}
(i)  The map $\tilde{\phi}$ gives a bijection between the set of reduced (nondegenerate) cubes and the set of  $(Q, \beta)$ on $E^2$ so that the standard basis $\{e_1, e_2\}$ is reduced. 
 \vskip 5pt
 
 (ii) The image of $\tilde{\phi}$ consists precisely of those $(Q,\beta)$ such that the standard basis $\{e_1, e_2\}$ of $C = E^2$ is a good basis for $(Q,\beta)$. 
\end{cor}
\vskip 5pt

The definition we have given for a good basis $\{e_1, e_2\}$ may not seem very satisfactory. It would have been  more satisfactory if one defines a good basis for $(C, Q,\beta)$ using purely the forms $(Q,\beta)$  rather than using the action of $\Aut_E(C)^0$. Indeed, it will not be easy to check that a given basis is good by our definition. 
However, by Corollary \ref{C:reduced}, one knows {\em a posteriori} that a basis $\{e_1, e_2\}$ is good for $(C,Q,\beta)$ if and only if $\beta(xe_1 +ye_2)$ has the form given in (\ref{E:Qb}) with $a, b \in F$.  
We would have taken this as a definition, but it would have seemed completely unmotivated without the results of this section!

\vskip 5pt
 
\subsection{\bf A commutative diagram}
 As a summary of the above discussion,  we have the following refinement and explication of Theorem \ref{T:main-orbit}:
 \vskip 5pt
 
 \begin{thm}  \label{T:main-orbit2}
(i)  The bijective map $\tilde{\phi}$ descends to a  give  a commutative diagram:
\[
\begin{CD}
 V_E(F)^0 = \{  c \in V_E(F): \Delta_E(c)  \ne 0\}  @>>> \text{$\tilde{M}_E(F)$-orbits on $V_E(F)^0$}   \\
@V\tilde{\phi}VV    @VV{\phi}V \\
 \{ \text{pairs $(Q,\beta)$ on $E^2$: standard basis is good}  \}  @>>>  \{ \text{$\GL_2(E)^0 \rtimes S_E(F)$-orbits of $(Q,\beta)$} \}   \\
 @VVV   @VVV \\
 \{ \text{$F$-isom. classes of pairs $(C,b)$} \}  @>>> \{ \text{$F$-isom. classes of $C$} \} 
  \end{CD} \]
 where all vertical arrows are $\GL_2(E)^0 \rtimes S_E(F)$-equivariant bijections and, in the last row, $C$ denotes an $E$-twisted composition algebra and $b$ denotes a good basis of $C$. Moreover the action of $\GL_2(E)^0 \rtimes S_E(F)$ on a pair $(C,\{e_1, e_2\})$ is given as follows: 
 $g \in \GL_2(E)^0$  sends it to  $(C,  \{ e_1', e_2'\})$, where
 \[  \left( \begin{array}{cc}
 e_1' \\
 e_2' \end{array} \right)  = g \cdot 
 \left( \begin{array}{cc}
 e_1 \\
 e_2 \end{array} \right),
 \]
 whereas $\sigma \in S_E$ sends it to $(E \otimes_{E,\sigma} C,  \{ e_1, e_2\})$. 
 \vskip 5pt

 \vskip 5pt

 \noindent (ii) The bijection $\phi$ agrees with the one given in  Theorem  \ref{T:main-orbit}.
 
  \end{thm}
\vskip 5pt

\begin{proof}
(i) Our discussion above already shows that $\tilde{\phi}$ is bijective and descends to give the map $\phi$. It remains to show that the induced map $\phi$ is bijective.
The surjectivity of $\phi$ follows from Proposition \ref{P:reduced-basis}(iii) and (iv),  and Corollary \ref{C:reduced}(i).
The injectivity of $\phi$ follows from Proposition \ref{P:bij}(i) and (ii). We leave the bijection and the equivariance of the lower half of the diagram to the reader.
\vskip 5pt

\noindent (ii) The map $\tilde{\phi}$ sends the distinguished cube $v_{E,0} =(1,0,0,-1)$ to the pair $(Q_0, \beta_0)$ on $E^2$, which defines the algebra $C_E$. Moreover, $\tilde{\phi}$ 
is equivariant with respect to the automorphism  $g \mapsto {^t}g^{-1}$  of $\GL_2(E)$, which preserves the subgroup  ${\rm Stab}_{\GL_2(E)^0}(v_{E,0}) = \Aut_E(Q_0,\beta_0) \subset \GL_2(E)^0$.  Finally, since $\tilde{\phi}$ is algebraic, it is  Galois-equivariant with respect to base field extension.  All these imply that we have a commutative diagram
\[  \begin{CD}
\{\text{$\GL_2(E)^0$-orbits on $V_E(F)^0$} \}@>>> H^1(F,  {\rm Stab}_{\GL_2(E)^0}(v_{E,0}))  \\
@VV{\phi}V    @VV{g \mapsto {^t}g^{-1}}V \\
\{\text{$E$-isom. classes of twisted composition algebras} \}@>>> H^1(F,  \Aut_E(C_E))  \end{CD} \]
 Since the map $g \mapsto {^t}g^{-1}$ of  ${\rm Stab}_{\GL_2(E)^0}(v_{E,0}) = \Aut_E(Q_0,\beta_0)$ is given by conjugation by the element $w \in  \Aut_E(Q_0,\beta_0)(F)$, we see that the induced map on $H^1$ is trivial. Hence $\phi$ agrees with the bijection given in Theorem \ref{T:main-orbit} by Galois cohomological argument. 
 \end{proof}

\subsection{\bf An example.}
 As an example, 
 assume that $K=F(\sqrt{\Delta})$ and consider the composition algebra given by the example in  \S \ref{SS:example2}. 
 (This  is the distinguished point in the fiber of $([F^3], [K])$.) Then $v=(\sqrt{\Delta}, \sqrt{\Delta}, \sqrt{\Delta})$ and 
 $\beta(v)=(\Delta,\Delta, \Delta)$ is a reduced basis. The corresponding reduced cube is 

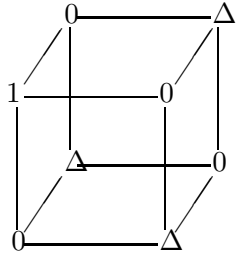
\begin{picture}(100,130)(-130,10)

\put(20,18){$0$}
\put(25,20){\line(1,0){50}}
\put(76,18){$\Delta$}

\put(40,48){$\Delta$}
\put(45,50){\line(1,0){50}}
\put(96,48){$0$}

\put(22,26){\line(0,1){46}}
\put(78,26){\line(0,1){46}}

\put(42,56){\line(0,1){46}}
\put(98,56){\line(0,1){46}}

\put(25,25){\line(2,3){14}}
\put(82,25){\line(2,3){14}}

\put(18,74){$1$} 
\put(25,76){\line(1,0){50}}
\put(76,74){0}

\put(40,104){0}
\put(45,106){\line(1,0){50}}
\put(96,104){$\Delta$}

\put(25,81){\line(2,3){14}}
\put(82,81){\line(2,3){14}}

\end{picture}
 
 \noindent 
\vskip 15pt

 \subsection{\bf Relation with Tits' construction.}   
   If $f\in E^{\times}$, we can relate the construction of $\tilde{\phi}$  attached to the reduced cube $(1,0,-f,b)$ to Proposition \ref{P:kmrt-quad}.   Identify $E\oplus E$ with  $E\otimes K$  using the $E$-linear isomorphism given by 
    \[ 
   (x,y)\mapsto  x\otimes 1 + \frac{y}{f} \otimes  \frac{b- \sqrt{\Delta}}{2} = x + y \frac{b- \sqrt{\Delta}}{2f}
   \] 
 where, in the last expression, we omitted tensor product signs for readability. Then $Q$ can be written as    
 \[ 
   Q( x+y \frac{b- \sqrt{\Delta}}{2f})= -f  \cdot N_{E\otimes K/E} ( x+ y \frac{b- \sqrt{\Delta}}{2f})
   \] 
 and $\beta$ as 
  \[ 
  \beta ( x+ y \frac{b- \sqrt{\Delta}}{2f})= -\frac{2}{b+\sqrt{\Delta}}\cdot f^{\#} \cdot ( x+ y \frac{b+ \sqrt{\Delta}}{2f})^{\#}. 
  \] 
  Indeed, if $E=F^3$, these formulae are exactly the same as those in subsection \ref{SS:integral} 
.  Let 
  \[  e=-f \quad \text{and}  \quad \nu=-\frac{b+\sqrt{\Delta}}{2}.  \]
  Using $e^{-1} \cdot  \bar\nu =\nu^{-1} \cdot  e^{\#}$ (since $N_{E/F}(e)=N_{K/F}(\nu)$) this composition algebra is the algebra attached to the pair $(e,\nu)$, as  in Proposition \ref{P:kmrt-quad}. 
   Conversely, a composition algebra 
  given by a pair $(e,\nu)$, as  in Proposition \ref{P:kmrt-quad}, arises from the cube $(1,0,-e, b)$ where $b=-Tr_{K/F}(\nu)$.

 \vskip 10pt
 



\section{\bf Exceptional Hilbert 90} \label{S:H90}

Assume that $E$ is an \'etale cubic $F$-algebra  field with corresponding \'etale quadratic discriminant algebra $K_E$ and let $K$ be an \'etale quadratic $F$-algebra.  
Recall that 
\[  T_{E,K} = \{ x \in E\otimes_F K: N_{E/F}(x) = 1 = N_{K/F}(x) \}. \]
Suppose,
  for example, that  $[K_E]= [K]=1$, so $E$ is a Galois extension, and $T_{E,K}$ is the group of norm one 
elements in $E^{\times}$. Let $\sigma$ be a generator of the Galois group $G_{E/F}$.  Then Hilbert's Theorem 90 states that  the map 
\[ 
x\mapsto \sigma(x)/\sigma^2(x) 
\] 
induces an isomorphism of $E^{\times}/F^{\times}$ and $T_{E,K}(F)$. Our goal in this section is to generalize this statement to all 
tori $T_{E,K}$, thus obtaining an exceptional Hilbert's Theorem 90. As an application, we give an alternative description of $H^1(F, T_{E,K})$. 
\vskip 10pt

\subsection{\bf The torus $T_{E,K}$.}
We first describe the torus $T_{E,K}$ by Galois descent. 
Over $\overline{F}$, we may identify
\[  T_{E,K}(\overline{F}) = \{ (\underline{a}, \underline{b}) \in   \overline{F}^3 \otimes  \overline{F}^2:
 \text{$a_ib_i=1$ for all $i$ and $a_1a_2a_3 = 1$} \}. \]
The $F$-structure is given by the twist of the Galois action on coordinates by the cocycle
  \[  \rho_E  \times  \rho_K : \Gal(\overline{F}/F) \longrightarrow \Aut(\overline{F}^3) \times \Aut(\overline{F}^2) \cong S_3 \times \Z/2\Z,\]
 where $S_3$ (respectively  $\Z/2\Z$) acts on $\Z^3$ (respectively $\Z^2$) by permuting the coordinates.

\vskip 5pt

We may describe $T_{E,K}$ using its  cocharacter lattice $X$. We have:
 \[  X  =  \{  (\underline{a}, -\underline{a}) \in \Z^3 \otimes \Z^2:  a_1+a_2+a_3 = 0 \}, \]
 equipped with the Galois action given by
 \[  \rho_E \otimes \rho_K : \Gal(\overline{F}/F) \longrightarrow S_3 \times \Z/2\Z. \]
 \vskip 5pt
 
\subsection{\bf The torus $T'_{E,K}$.}
Now we introduce another torus $T'_{E,K}$ over $F$.
 Let $K_J$ be the \'etale quadratic  $F$-algebra such that $[K_J]\cdot [K]\cdot [K_E]=1$ in $H^1(F, \Z/2\Z)$.  We define the tori
 \[  \tilde{T}'_{E,K} = \{ x \in E \otimes_F K_J:   N_{E \otimes K_J / E}(x) \in F^{\times} \}, \]
 and 
 \[   T'_{E,K}  = \tilde{T}'_{E,K}  / K_J^{\times} \]  
 where the last quotient is taken in the sense of algebraic groups.  If $J=B^{\tau}$ where $B$ is a degree 3 central simple $K_J$-algebra with an involution $\tau$ of the second kind, 
 and $E \rightarrow J$ is an $F$-embedding or, equivalently, $E\otimes_F K_J \rightarrow B$ is a $K_J$-embedding such that $\tau$ pulls back to the nontrivial element of 
 $\Aut(E\otimes_F K_J/ E)$, then $T'_{E,K}$ acts naturally as a group of automorphisms of the embedding $E\rightarrow J$. 
\vskip 5pt

We may again describe these tori by Galois descent.
 Over $\overline{F}$,  we may identify
 \[  \tilde{T}'_{E,K}(\overline{F}) = \{ (\underline{a},  \underline{b}) \in  (\overline{F}^{\times})^3 \otimes  (\overline{F}^{\times})^2:  a_1b_1 = a_2b_2 = a_3b_3 \},  \] 
 and $T'_{E,K}(\overline{F})$ is the quotient of this by the subgroup consisting of the elements $(a \cdot \underline{1}, b \cdot \underline{1} )$. 
 The action of $\Gal(\overline{F}/F)$ which gives the $F$-structure of $\tilde{T}'_{E,K}$ is then described as follows. Let $\rho_E:  \Gal(\overline{F}/F) \longrightarrow S_3$ be the  cocycle  associated to $E$, so that 
 ${\rm sign} \circ  \rho_E: \Gal(\overline{F}/F) \longrightarrow \Z/2\Z$ is the homomorphism associated to $K_E$. On the other hand, we let $\rho_K$ be the homomorphism  associated to $K$, so that
 \[   ({\rm sign} \circ \rho_E) \cdot \rho_K  : \Gal(\overline{F}/F) \longrightarrow \Z/2\Z \]
 is the homomorphism associated to $K_J$.
  Now the action of $\Gal(\overline{F}/F)$ on 
 $ \overline{F}^3 \otimes  \overline{F}^2$ is the twist of the action on coordinates by the cocycle
 \[  \rho_E  \times  ({\rm sign} \circ \rho_E) \cdot \rho_K  : \Gal(\overline{F}/F) \longrightarrow S_3 \times \Z/2\Z. \]
  
 \vskip 5pt
 As before, we may describe the tori $\tilde{T}'_{E,K}$ and $T'_{E,K}$ by their cocharacter lattice.
 The cocharacter lattice $\tilde{Y}$ of $\tilde{T}'_{E,K}$ is given by
 \[  \tilde{Y}  = \{  (\underline{a}, \underline{b}) \in \Z^3 \otimes \Z^2:  a_1+ b_1 = a_2 + b_2 = a_3 + b_3\},  \]
 equipped with the Galois action given by
  \[  \rho_E  \times ({\rm sign} \circ \rho_E) \cdot \rho_K :  \Gal(\overline{F}/F) \longrightarrow S_3 \times \Z/2\Z. \]
 This contains the Galois-stable sublattice 
\[  Z  = (1,1,1) \otimes \Z^2    \]
so that $Y = \tilde{Y}/Z$ is the cocharacter lattice of $T'_{E,K}$.

 \vskip 5pt

 \subsection{\bf A homomorphism.}
 We are going to construct a morphism of tori from $\tilde{T}'_{E,K}$ to $T_{E,K}$. We shall first define this morphism over $\overline{F}$ and then shows that it descends to $F$. 
 \vskip 5pt
 
   Now we may define a morphism over $\overline{F}$:
  \[  f: \tilde{T}'_{E,K}(\overline{F})  \longrightarrow T_{E,K}(\overline{F}) \]
  by
  \[   f: \left( \begin{array}{ccc}
 a_1 & a_2 & a_3 \\
 b_1 & b_2 & b_3 \end{array} \right)
  \mapsto 
   \left( \begin{array}{ccc}
  a_2/a_3 & a_3/a_1 & a_1/a_2 \\
  b_2/b_3 & b_3/b_1 & b_1/b_2 \end{array} \right) \]
  It is easy to see that this defines an $\overline{F}$-isomorphism of tori 
  \[  f:  T'_{E,K}(\overline{F}) \cong T_{E,K}(\overline{F}). \]
  Moreover, if $\sigma \in S_e(\overline{F})  =S_3$ is the cyclic permutation
  \[  (a_1, a_2, a_3) \mapsto (a_2, a_3, a_1), \]
  then the map $f$ is given by
  \[  f(x)  = \sigma(x) / \sigma^2(x). \]
  \vskip 5pt
  
   Now the morphism $f$ induces a map
 \[  f_*:   \tilde{Y} \longrightarrow   X \]
 given by
 \[   \left( \begin{array}{ccc}
 a_1 & a_2 & a_3 \\
 b_1 & b_2 & b_3 \end{array} \right)
  \mapsto 
   \left( \begin{array}{ccc}
  a_2-a_3 & a_3 -a_1 & a_1 - a_2 \\
  b_2- b_3 & b_3- b_1 & b_1 - b_2 \end{array} \right). \]
  This induces an isomorphism of $\Z$-modules $Y \cong X$. 
\vskip 5pt

  \subsection{\bf Exceptional Hilbert 90.} 
  Now the main result of this section is:
 \vskip 5pt
 
 \begin{thm}
 The isomorphism $f:  T'_{E,K} \times_F \overline{F}  \longrightarrow T_{E,K} \times_F \overline{F}$
 is defined over $F$, and thus gives 
 an isomorphism of tori
 \[  T'_{E,K} \longrightarrow T_{E,K} \]
 given by
 \[  x \mapsto \sigma(x)/  \sigma^2(x). \]
 \end{thm}
 \vskip 5pt

 \begin{proof}
 It remains to prove that $f$ is defined over $F$. For this, we may work at the level of cocharacter lattices and we need to show that $f_*$ is Galois-equivariant.  For this, regard $\Z^3 \otimes \Z^2$ as a $S_3 \times \Z/2\Z$-module with the permutation of the coordinates in $\Z^3$ and $\Z^2$. Then observe that $f_*$ is not equivariant with respect to $S_3 \times \Z/2\Z$. On the other hand, we have the automorphism of $S_3 \times \Z/2\Z$ given by
  \[  (g, h) \mapsto (g, {\rm sign}(g) \cdot h) \]
  If we twist  the $S_3 \times \Z/2\Z$-module structure on the domain of $f_*$ by this automorphism, then $f_*$ is easily seen to be equivariant.  Together with our description of the $\Gal(\overline{F}/F)$-actions on the domain and codomain of $f_*$, the desired $\Gal(\overline{F}/F)$-equivariance follows. 
 \end{proof}
 \vskip 10pt

 \subsection{\bf Cohomology of $T_{E,K}$.} 
 As an application of the exceptional Hilbert 90, we may give an alternative description of the cohomology group $H^1(F, T_{E,K})$ which classifies twisted composition algebras with fixed invariants $(E, K)$, up to $E \otimes_F K$-linear isomorphisms. 
 
 \vskip 5pt
 
  In order to state results,  we need additional notation. For every quadratic extension $K_J$ of $F$,  let $\Res^1_{K_J/F} \mathbb{G}_m$ be the 1-dimensional torus defined by the short exact sequence of algebraic tori:
  \[  \begin{CD}
  1 @>>> \Res^1_{K_J/F} \mathbb G_m @>>> \Res_{K_J/F} \mathbb G_m @>>> \mathbb{G}_m @>>> 1.\end{CD} \]
 By   the classical Hilbert Theorem 90,   the associated long exact sequence gives the exact sequence: 
 \[ 
 1\longrightarrow  H^2(F, \Res_{K_J/F}^1 \mathbb{G}_m)  \longrightarrow  H^2(K_J, \mathbb G_m) \longrightarrow H^2(F, \mathbb G_m) 
 \] 
 where the last map is the corestriction. By a theorem of Albert and Albert-Riehm-Scharlau, Theorem 3.1 in \cite{KMRT}, the kernel of the corestriction map is the set of Brauer equivalence classes of central 
 simple algebras over $K_J$ that admit an involution of the second kind, and so we can view $H^2(F, \Res_{K_J/F}^1\mathbb{G}_m)$ as the set of Brauer equivalence classes of such algebras. 
 \vskip 5pt
 
 Now we have:
 \begin{prop}  \label{P:coho-T2}
 Let $K_J$ be the \'etale quadratic algebra such that $[K_J] \cdot [K] \cdot [K_E] =1$ and set $M = E \otimes_F K_J$. 
 \vskip 5pt
 
\noindent  (i) If $K_J$ is a field, then we have an exact sequence 
\[ 
 1\longrightarrow E^{\times}/F^{\times} N_{M/E}(M^{\times}) \longrightarrow  H^1(F, T_{E,K})   \longrightarrow  H^2(F, \Res_{K_J/F}^1 \mathbb G_m) \longrightarrow H^2(E, \Res_{M/E}^1 \mathbb{G}_m)  \] 
 The image of $H^1(F, T_{E,K})$ consists of those central simple algebras $B$ over $K_J$ which contain $M$ as a $K_J$-subalgebra and which admit an involution of the second kind fixing $E$ (or equivalently, restricting to the nontrivial automorphism of $M$ over $E$). 
 \vskip 5pt

\noindent (ii) If $K_J = F^2$, then we have a simplified version of the above sequence 
\[  H^1(F, T_{E,K}) = {\rm Ker}(H^2(F, \mathbb{G}_m)  \longrightarrow H^2(E, \mathbb{G}_m)). \]
  \end{prop}
 
 \vskip 5pt
 
 \begin{proof}
 (i) By the exceptional Hilbert Theorem 90, we have a short exact sequence of algebraic tori:
 \[  \begin{CD}
 1@>>> \Res^1_{K_J/F}\mathbb  G_m @>>> \Res_{E/F} \Res^1_{M/E} \mathbb G_m  @>>> T_{E,K} @>>>1. \end{CD} \]
   Now (i) follows from the associated long exact sequence, using 
 \[  H^1(F, \Res^1_{K_J/F}\mathbb  G_m )  = F^{\times}/N_{K_J/F} K_J^{\times} \quad \text{and} \quad H^1(E, \Res^1_{M/E} \mathbb G_m )  = E^{\times}/N_{M/E} M^{\times}. \]
 
\vskip 5pt

\noindent (ii) One argues as above, except that since $K_J  = F^2$, we have:
\[  \begin{CD}
 1@>>> \mathbb  G_m@>>> \Res_{E/F} \mathbb G_m @>>> T_{E,K} @>>>1. \end{CD} \]
 Thus the long exact sequence gives
 \[  \begin{CD}
 1 @>>> H^1(F, T_{E,K})  @>>> H^2(F, \mathbb{G}_m)  @>>> H^2(E, \mathbb{G}_m) \end{CD} \]
  \end{proof}
 \vskip 5pt
 
 \subsection{\bf Interpretation.}
  The above description of $H^1(F, T_{E,K})$ fits beautifully with the correspondence between $E$-twisted composition algebras and conjugacy classes of embeddings $E \hookrightarrow J$ where $J$ is a Freudenthal-Jordan algebra of dimension $9$. 
  \vskip 5pt
  
  More precisely, Proposition \ref{P:coho-T2} exhibits  $H^1(F, T_{E,K})$ as the set of isomorphism classes of triples $(B, \tau, i)$ where
  \vskip 5pt
  
  \begin{itemize}
  \item $B$ is a central simple $K_J$-algebra of degree $3$;
  \item $\tau$ is an involution of the second kind on $B$;
  \item $i:  E \longrightarrow B^{\tau}$ is an $F$-algebra embedding, or equivalently a $K_J$-algebra embedding $i: M  = E \otimes_F K_J \longrightarrow B$ such that $\tau$ pulls back to the nontrivial element of $\Aut(M/E)$.
  \end{itemize}
  \vskip 5pt
  
\noindent The map $\pi: H^1(F, T_{E,K}) \longrightarrow H^2(F, \Res_{K_J/F}^1 \mathbb G_m)$ sends
$(B, \tau, i)$ to $B$. For a fixed 
\[  [B] \in {\rm Ker}(H^2(F, \Res_{K_J/F}^1 \mathbb G_m) \longrightarrow H^2(E, \Res_{M/E}^1 \mathbb{G}_m)), \]
so that $B$ contains $M =E\otimes_F K_J$ as an $K_J$-subalgebra, 
the fiber of $\pi$ over $[B]$ is the set of $\Aut_{K_J}(B)$-conjugacy classes of pairs $(\tau, i)$.
The Skolem-Noether theorem says that any two embeddings $M \hookrightarrow B$ are conjugate, and on fixing an embedding $i: M \hookrightarrow B$, the fiber of $\pi$ over $[B]$ is then the set of  $\Aut_{K_J}(B, i)$-conjugacy classes of involutions of the second kind on $B$ which restricts to the nontrivial automorphism of $M$ over $E$.  
Therefore, the exact sequence in Proposition \ref{P:coho-T2}(i) says that the set of such $\Aut_{K_J}(B,i)$-conjugacy classes of involutions is identified with $E^{\times}/F^{\times} N_{M/E}(M^{\times})$. One has a natural map on the fiber $\pi^{-1}([B])$  sending a $\Aut_{K_J}(B,i)$-conjugacy class of involutions to its $\Aut_{K_J}(B)$-conjugacy class. This is the surjective map described in Corollary 19.31 in \cite{KMRT}.

 \vskip 5pt

On the other hand, the map sending the triple $(B, \tau, i)$ to the pair $(B,\tau)$ is the natural map
\[  H^1(F, T_{E,K}) \longrightarrow H^1(F, PGU_3^{K_J})  \]
induced by the map $T_{E,K} \hookrightarrow PU_3^{K_J} $ where $PGU_3^{K_J}$ is the identity component of the automorphism group of the Freuthendal-Jordan algebra associated to the distinguished twisted composition algebra with invariants $(E,K)$. 
  \vskip 15pt

 \section{\bf Local Fields}
 In this section, we specialize and explicate the main result in the case of local fields.
 \vskip 5pt
 
 \subsection{\bf Local Fields.}
 Let $F$ be a local field, $E$ an \'etale cubic $F$-algebra, and $K_E$ the corresponding discriminant  algebra. Let $K$ be an \'etale quadratic $F$-algebra.  
 We consider
 \[  \tilde{\Omega}_{E,K} = \{\text{generic $\tilde{M}_E$-orbits on $V_E$ with associated quadratic algebra $K$} \} \]
 and
 \[  \Omega_{E,K} = \{\text{generic $M_E$-orbits on $V_E$ with associated quadratic algebra $K$} \} \]
 We have seen that $\tilde{\Omega}_{E,K}$ has a distinguished element: this is the distinguished point of $H^1(T_{E,K})$ which is fixed by $S_E(F) \times \Z/2\Z$.
 Moreover,  by Galois cohomological arguments, 
 \[ \tilde{\Omega}_{E,K}  = H^1(F, T_{E,K}) / S_E(F) \times \Z/2\Z \quad \text{and} \quad \Omega_{E,K} = H^1(F, T_{E,K}) /\Z/2\Z  \]
  We would like to explicate the sets $\tilde{\Omega}_{E,K}$ and $\Omega_{E,K}$.  
 \vskip 5pt
 
 \subsection{Cohomology of tori.} 
 Recall that in (\ref{E:coho-T}), we have shown
 \[  H^1(F, T_{E,K})  =  (E^{\times} \times K^{\times})^0 /  {\rm Im} (L^{\times}) \]
 where $L = E \otimes_F K$,
 \[ (E^{\times} \times K^{\times})^0 = \{ (e, \nu) \in E^{\times} \times K^{\times}: N_{E/F}(e)  = N_{K/F}(\nu) \} \]
 and the map from $L^{\times}$ to $(E^{\times} \times K^{\times})^0$ is given by
 \[  a \mapsto (N_{L/E}(a), N_{L/K}(a)). \]
This description of $H^1(F, T_{E,K})$ is natural but may not be so explicit.
 When $F$ is a local field, we can further explicate this description. 
 \vskip 5pt
 
 Since the case when $E$ or $K$ is not a field is quite simple, we consider the case when $E$ and $K$ are both fields. In that case, the norm map induces an isomorphism
 \[  E^{\times}/ N_{L/E}(L^{\times}) \longrightarrow   F^{\times}/ N_{K/F}(K^{\times}) \cong \Z/2\Z, \]  
 so that any $(e ,\nu) \in  (E^{\times} \times K^{\times})^0$ has  $e =  N_{L/E}(a)$ for some $a \in L^{\times}$. Hence any element in $H^1(F, T_{E,K})$ is represented by $(1, \nu)$ for some $\nu \in K^1 = \{ \nu \in K^{\times}: N_{K/F}(\nu) = 1\}$. We thus deduce that, with $L^1 = \{  a \in L^{\times}:  N_{L/E}(a)  =1\}$,
 \[  H^1(F, T_{E,K})  = K^1/  N_{L/K}(L^1)  \cong  K^{\times}/F^{\times}N_{L/K}(L^{\times}),
 \]
where the last isomorphism is induced by the usual Hilbert Theorem 90. 
Using this last expression, we easily see that
\[  H^1(F, T_{E,K})  = \begin{cases} 
1 \text{  if $K \ne K_E$;} \\
\Z/3\Z, \text{  if $K = K_E$.} \end{cases} \]
 \vskip 5pt
 
 Exchanging the roles of $E$ and $K$ in the above argument, one also has:
  \[  H^1(F, T_{E,K})  = E^1/  N_{L/E}(L_1) \]
  where now $L_1 = \{ a \in L^{\times}: N_{L/K}(a)  =1 \}$.  If $E/F$ is Galois (and $K$ is a field), it follows by the usual Hilbert Theorem 90 that
  \[  H^1(F, T_{E,K}) = E^1/  N_{L/E}(L_1)  \cong E^{\times}/F^{\times} N_{L/E}(E^{\times}) = 1, \]
thus partially recovering the result of the last paragraph.
 \vskip 5pt
 
 Alternatively, we could use Proposition \ref{P:coho-T2} to compute $H^1(F, T_{E,K})$. If $K_J$ is a field, then the only central simple $K_J$-algebra which admits an involution of the second kind is the split algebra $M_3(K_J)$. Thus we deduce from Proposition \ref{P:coho-T2}(i) that
 \[  H^1(F, T_{E,K}) \cong  E^{\times}/F^{\times} N_{M/E}(M^{\times})  \]
 where $M = E \otimes_F K_J$. On the other hand, if $K_J$ is split, then   Proposition \ref{P:coho-T2}(ii) gives
 \[  H^1(F, T_{E,K}) \cong  {\rm Ker}(H^2(F, \mathbb{G}_m)  \longrightarrow H^2(E, \mathbb{G}_m))  \]
 which is $\Z/3\Z$ when $E$ is a field.
  \vskip 10pt

 \subsection{\bf Fibers.}
 
 With the various computations of $H^1(F, T_{E,K})$ given above, it is now not difficult to show the following proposition which determines $|\tilde{\Omega}_{E,K}|$ and $|\Omega_{E,K}|$.
  \vskip 5pt
 
 \begin{prop}
 We have
 \begin{center}
 \begin{tabular}{|c|c|c|c|c|c|}
 \hline
 $E$ & $K$ & $T_{E,K}$ & $H^1(F, T_{E,K})$ & $|\tilde{\Omega}_{E,K}|$ & $|\Omega_{E,K}|$ \\
 \hline 
 $F \times K_E$ &  $K = K_E$ &   $K^{\times}$ & $1$&  $1$ & $1$ \\
 \hline 
 $F \times K_E$, $K_E$ a field &  field$\ne K_E$  & $(K \otimes K_E)^{\times}/ K_E^{\times}$ & $ \Z/2\Z$ & $2$ & $2$ \\
 \hline 
 $F \times K_E$, $K_E$ a field & $F \times F$ & $K_E^{\times}$ & $1$ & $1$ & $1$\\
 \hline 
 $F ^3$  &  field  &  $K^{\times}/F^{\times} \times K^{\times}/F^{\times}$ & $ \Z/2\Z \times \Z/2\Z$ & $2$ 
 & $4$ \\
 \hline 
  field & $K = K_E$  &  $E^{\times}/ F^{\times}$ & $\Z/3\Z$ & $2$ & $2$  \\
 \hline 
 field & $K \ne K_E$ &   & $1$ & $1$ & $1$ \\  \hline
 \end{tabular}
 \end{center}
 
 Here, the difference in the last two columns reflects the fact that $S_E(F)$ acts trivially on $H^1(F, T_{E,K})$ except when $E = F^3$ and $K$ is a field.
 \end{prop}
 
  \vskip 5pt

 \subsection{\bf Embeddings into $J$.} 
 The main theorem says that the elements of $\Omega_{E,K}$   are in bijection with the conjugacy classes of embeddings
 \[    E \hookrightarrow J \]
 where $J$ is a 9-dimensional Freudethal-Jordan algebra associated to 
 a pair $(B, \tau)$ where $B$ is a central simple algebra over the quadratic algebra $K_J$ and $\tau$ is an involution of the second kind on $B$.  
We now describe the elements of $\Omega_{E,K}$ in terms of such embeddings. 
 
  \vskip 5pt
 
 \begin{itemize}
 \item when $F$ is p-adic and $K = K_E$ so that $K_J = F \times F$ is split, then 
 \[   (B, \tau)  = (D \times D^{op}, {\rm sw})   \]
 where $D$ is a central simple $F$-algebra of degree $3$, and $sw$ denotes the involution which switches the two factors.  Thus, there are 2 possible $J$'s in this case:  the Jordan algebra $J^+$ attached to  $M_3(F)$ or  the Jordan algebra $J^-$ attached to  a cubic division $F$-algebra (and its opposite). In either case, the set of embeddings $E \longrightarrow J$ is either empty or a single conjugacy class, and it is empty if and only if $J = J^-$ and $E$ is not a field. Thus when $K = K_E$, we have:
 \[  \tilde{\Omega}_{E,K} = \Omega_{E,K} =\begin{cases}
  \{  E \rightarrow J^+, E \rightarrow J^- \}  \text{ if $E$ is a field;} \\
 \{  E \rightarrow J^+ \}, \text{  if $E$ is not a field.} \end{cases} \]
 \vskip 5pt
 
   On the other hand, when $K_J$ is a field, then $B  = M_3(K_J)$, and
 there is a unique isomorphism class of involution of the second kind on $B$, given by conjugation by a nondegenerate hermitian matrix, so that $J$ is isomorphic to the Jordan algebra of $3 \times 3$-Hermitian matrices with entries in $K_J$.   
 According the the proposition, there is a unique conjugacy class of embedding $E \hookrightarrow J$ unless $E  = F \times K_E$ and $K$ is a field with $K \ne K_E$. In the exceptional case, there are two subalgebras $E \subset J$ up to conjugacy.  We may write down the 2 non-$F$-isomorphic twisted composition algebras corresponding to these. The twisted composition algebra can be realised on
 \[   E \otimes_F K = K \times (K_E \otimes  K). \]
 Let $\{1, \alpha \}$ denote representatives of $F^{\times}/NK^{\times}$.  Then the 2 twisted composition algebras correspond to
 \[   (e, \nu)  = ((1,1), 1)  \quad \text{or} \quad ((1, \alpha),   \alpha  ) \in (F \times K_E)^{\times} \times K^{\times}. \]
 We see that these two twisted composition algebras are not isomorphic because they are not isomorphic as quadratic spaces over $E$ (even allowing for twisting by $S_E(F)$).  
  \vskip 5pt
  
  Further, when $E = F^3$, there are in fact 4 conjugacy classes of embeddings $E \hookrightarrow J$. This corresponds to the fact that the $F$-isomorphism class of the twisted composition algebras associated to $((1,\alpha), \alpha)$ above breaks into $3$ $E$-isomorphism classes. These are associated to 
  \[  (e_1, \nu_1)  = ((1, \alpha, \alpha), \alpha),\quad  e_2 = ((\alpha, 1, \alpha),\alpha), \quad
   e_2 = ((\alpha, \alpha, 1),\alpha). \]
  \vskip 5pt
 
 \item when $F = \mathbb{R}$, then $E = \mathbb{R}^3$ or $\mathbb{R} \times \mathbb{C}$. 
 When $K_J = \mathbb{R}^2$ is split, then there is a unique $J$, namely the one associated to 
 $M_3(\mathbb{R})$, and there is a unique conjugacy class of embeddings $E \hookrightarrow J$. 
 \vskip 5pt
 
  When $K_J = \mathbb{C}$, then there are two possible $J$'s, associated to 
 $B = M_3(\mathbb{C})$ and the involution $\tau$ given by the conjugation action of two Hermitian matrices with signature $(1,2)$ and $(3,0)$. We denote these two Jordan algebras by $J_{1,2}$ and $J_{3,0}$.  
 \vskip 5pt

 When $E = \mathbb{R}^3$ and $K = \C$, we have $|\Omega_{E,K}|  =2$. However, the two elements in question correspond to embeddings
 \[  \R^3 \hookrightarrow J_{3,0} \quad \text{and} \quad \R^3 \hookrightarrow J_{1,2}. \]
  Thus, we see that these subalgebras are  unique up to conjugacy.
 When $E = \R \times \C$ and $K = \R^2$, we have $|\Omega_{E,K}| = 1$. This reflects the fact that there is no embedding $\R \times \C \hookrightarrow J_{3,0}$, and there is a unique conjugacy class of embeddings $\C \hookrightarrow J_{1,2}$.
 \end{itemize}

 \subsection{Acknowledgment.} This work began when both authors participated in the program ``Branching Laws" at the Institute for Mathematical Sciences at the National University of Singapore in March 2012, and  was completed during the second author's stay at The Hong Kong University of Science and Technology in May 2013. Both authors thank Chengbo Zhu for his invitation to the IMS program.   The second author would like to thank Jianshu Li for the invitation, and HKUST for excellent working environment.  
 \vskip 5pt
 
 The first author is partially supported by AcRF Tier One grant R-146-000-155-112, and the second author is  supported by  a National Science Foundation grant  DMS-0852429.

\end{document}